\documentclass[12pt]{article}
\usepackage{amssymb}
\usepackage{amsmath}
\newtheorem{prop}{Proposition}[section]
\newtheorem{theor}{Theorem}[section]
\newtheorem{lemma}{Lemma}[section]

\newtheorem{example}{Example}[section]
\newtheorem{remark}{Remark}[section]

\numberwithin{equation}{section}

\newenvironment{proof}[1][Proof]{\noindent\textbf{#1.}}{\hfill$\Box$}

\newcommand{\n}{\hbox{$\scriptstyle\hbox{\rm
      l\kern-0.22em N}$}}
\newcommand{\N}{\hbox{$ I\kern -0.23em N$}}

\newcommand{\R}{\mathbb{R}}

\def\m{\noalign{\medskip}}

\def\PS #1 #2{\langle #1, #2 \rangle}

\def\1{1\!\!1}

\def\omegb{{\overline \Omega}}
\def\domeg{{\partial \Omega}}

\def\xb{{\bar x}}
\newcommand{\vfi}{\varphi}
\def\de{\delta}
\def\la{\lambda}
\def\ga{\gamma}
\def\al{\alpha}
\def\be{\begin{equation}}
\def\ee{\end{equation}}
\def\rife#1{(\ref{#1})}

\def\dd{d}

\begin{document}

\title{On the Large Time Behavior of Solutions of the Dirichlet problem for Subquadratic Viscous Hamilton-Jacobi Equations }
\author{Guy Barles\thanks{Laboratoire de
Math\'ematiques et Physique Th\'eorique (UMR 6083). F\'ed\'eration Denis Poisson (FR 2964) Universit\'e de Tours.
Facult\'e des Sciences et Techniques, Parc de Grandmont, 37200
Tours, France. E-mail address: barles@lmpt.univ-tours.fr.}
\and Alessio Porretta\thanks{Dipartimento di Matematica, Universit\`a di Roma Tor 
Vergata, Via della Ricerca Scientifica 1, 00133 Roma, Italia. E-mail address: porretta@mat.uniroma2.it.} \and
Thierry Tabet Tchamba$^*$\thanks{D\'epartement de Math\'ematiques, Facult\'e des Sciences,
Universit\'e de Yaound\'e I, BP: 812, Yaound\'e, Cameroun. E-mail address: tabet@lmpt.univ-tours.fr.}} 
 
\maketitle

\bigskip
\begin{small}
\noindent{\bf Abstract} In this article, we are interested  in the
large time behavior of solutions of the Dirichlet problem for
subquadratic viscous Hamilton-Jacobi Equations. In the
superquadratic case, the third author has proved that these
solutions can have only two different behaviors: either the
solution of the evolution equation converges to the solution of
the associated stationary generalized Dirichlet problem (provided that it
exists) or it behaves like $-ct+\varphi (x)$ where $c\geq0$ is a
constant, often called the ``ergodic constant" and $\varphi$ is a
solution of the so-called ``ergodic problem". 
In the present subquadratic case, we show that the situation is slightly more complicated: 
if the gradient-growth in the equation is like $|Du|^m$ with $m>3/2,$ 
then analogous results hold as in the superquadratic case, at least if $c>0.$  
But, on the contrary, if $m\leq 3/2$ or $c=0,$ then another different behavior 
appears since $u(x,t) + ct$ can be unbounded from below where $u$ is the solution 
of the subquadratic viscous Hamilton-Jacobi Equations.
\end{small}

\vspace{.5cm} {\small \noindent{\bf Key-words :} Viscous Hamilton-Jacobi
Equations, large time behavior, subquadratic case, Dirichlet problem, ergodic problem,
viscosity solutions.

\vspace{0.3cm}
\noindent{\bf AMS subject classifications : }
35K55, 35B40, 49L25}

\newpage
\tableofcontents
\section{Introduction}

We are interested in this work in the behavior, when $t \to +\infty,$ of the solution of the following initial-boundary value problem
\begin{eqnarray}
u_{t} - \Delta u + |Du|^m & = & f(x) \quad\hbox{in  }\Omega \times (0,+\infty) \label{vhje}\\
 u(x,0) & = & u_{0}(x) \quad \hbox{on  }\omegb \label{vhje-id}\\
u(x,t)  & =  & g (x) \quad \hbox{on  }\domeg \times (0, + \infty) \label{vhje-bc}
\end{eqnarray}
where $\Omega$ is a $C^2$- bounded and connected subset of $\R^N,$ $1< m \leq 2 $ and $f,
u_{0},g$ are real-valued continuous functions defined respectively
on $\omegb, \omegb$ and $\domeg.$ The boundary and initial data satisfy the following compatibility condition
\begin{equation}\label{compatibility}
u_{0}(x)=g(x) \quad\hbox{ for all } x\in \partial \Omega.
\end{equation}It is standard to show that this
problem has a unique solution $u : \omegb \times [0, + \infty) \to \R$ and, as long as
one does not need regularity properties, proofs of this fact are
easy by using viscosity solutions theory (see Barles and Da Lio \cite{BDL} and
references therein).

In  the superquadratic case ($m>2$) the study of the asymptotic behavior has been done by the third author 
in \cite{TTT} where it is shown that the solution $u$ can have only two different behaviors whether the equation
\begin{equation} \label{vhje-s}
- \Delta v + |Dv|^m  =  f(x) \quad\hbox{in  }\Omega\; ,
\end{equation} has {\em bounded} subsolutions or not. If (\ref{vhje-s}) has a
bounded subsolution, then there exists a solution $u_\infty$ of
the stationary Dirichlet problem, i.e. of (\ref{vhje-s}) together
with the generalized Dirichlet boundary condition
\begin{equation}\label{vhjes-bc}
v(x)  =   g(x) \quad \hbox{on  }\domeg \; ,
\end{equation}
and $u(x,t) \to u_\infty (x)$ uniformly on $\omegb.$ This is the
most expected behavior. On the other hand, it can happen that (\ref{vhje-s}) has
no bounded subsolution and, in this case, one has to introduce the
{\em ergodic problem} with state constraint boundary conditions, namely (the reader has to keep in mind that
we are here in the superquadratic case)
\begin{equation} \label{vhje-sc}
\left\{
\begin{array}{rcl}
- \Delta \varphi + |D \varphi|^m & = & f(x) +c  \quad\hbox{in   }\Omega\; ,\\
- \Delta \varphi + |D \varphi |^m & \geq & f(x) +c \quad\hbox{on  }\domeg\;.
\end{array}
\right.
\end{equation}
We recall that, in this type of problems, both the solution $\varphi$
and the constant $c$ (the ergodic constant) are unknown. The existence and uniqueness of
solutions $(c,\varphi)$ for (\ref{vhje-sc}) is studied in Lasry and
Lions \cite{LL} (see also \cite{TTT} for a viscosity solutions
approach): the constant $c$ is indeed unique while the solution $\varphi$
is continuous up to the boundary and unique up to an additive constant.

Concluding the reference when $m>2$, it is proved in \cite{TTT}  that, if \rife{vhje-s} has no bounded subsolution, then  $c>0$ and   the function $u(x,t) + ct$
converges uniformly on $\omegb$ to a solution $\varphi$ of
(\ref{vhje-sc}) when $t\to + \infty.$ In fact, in the superquadratic
case, even if the boundary condition reads $u(x,t) + ct = g(x) +ct$ on $\domeg$ with $g(x) +ct\to+\infty$ as $t\to+\infty,$ there is a loss of boundary condition (cf. \cite{BDL}) and $u(x,t) + ct$ remains bounded on $\omegb.$

The first key difference in the subquadratic case is 
that there is no loss of boundary conditions and $u(x,t) + ct$ is actually equal to  $g(x) +ct$ on the boundary. Formally this forces the limit of $u(x,t) + ct$ to tend to $+\infty$ on the boundary and, actually, the analogue of (\ref{vhje-sc}) is
\begin{equation} \label{vhje-sub}
\left\{
\begin{array}{rcl}
- \Delta \varphi + |D \varphi|^m &=&  f(x) +c  \quad\hbox{ in } \Omega\; ,\\
\varphi (x)  &\to&  +\infty \quad\hbox{ when  } x \to \domeg\; .
\end{array}
\right.
\end{equation}
This problem was also studied in \cite{LL} where it is proved 
 that there exists a unique constant $c$ such that
(\ref{vhje-sub}) has a solution and, as in the superquadratic
case, the solution is unique up to an additive constant.

Coming back to the asymptotic behavior of $u,$ it can be thought,
at first glance, that it is essentially the same as in the
superquadratic case, with problem (\ref{vhje-sc}) being replaced
by problem (\ref{vhje-sub}). Surprisingly this is not true for any
$1<m\leq 2$ nor for any $c.$ Part of the explanations concerning $c$
is that the equation
$$- \Delta \varphi + |D \varphi|^m  =  f(x) + \tilde c  \quad\hbox{ in   }\Omega\; ,$$
has subsolutions which are bounded on $\omegb$ for any $\tilde c \geq c$ 
in the superquadratic case but only for $\tilde c > c$ in the subquadratic one.

In this article we prove the following: first, if the stationary Dirichlet
problem (\ref{vhje-s})-(\ref{vhjes-bc}) has a solution, then $u$
converges uniformly in $\omegb$ to this solution. Otherwise we
show that necessarily $c\geq 0$ and $u$ can have different types of behavior. Observe that  a curious feature in the
subquadratic case (which does not occur in the superquadratic
case) is that $c$ can be equal to $0$ even if the stationary
Dirichlet problem has no solution; this can be seen as a corollary
of the above remark. The different behaviors of $u$ can be described as
follows.
\begin{itemize}
\item[(i)] If $3/2 < m \leq  2$ and $c>0$ then $u(x,t) + ct$
converges locally uniformly in $\Omega$ to $\varphi$ where
$\varphi$ is a solution of (\ref{vhje-sub}). 
\item[(ii)]  If $1 < m \leq 3/2$ and $c>0$ then $\displaystyle \frac{u(x,t)}{t} \to
-c$ locally uniformly in $\Omega$ but it can happen that $u(x,t)+ ct \to -\infty$ in $\Omega$.
 \item[(iii)] If $c=0 $, then
$ \displaystyle\frac{u(x,t)}{t} \to 0$ locally uniformly in $\Omega$ but it can 
happen that $u(x,t)\to -\infty$ in $\Omega$.
\end{itemize}

The behaviors (ii) and (iii) are striking differences with the superquadratic case. This is related to  the  blow-up rate of $u(x,t)$ which we estimate in Theorem \ref{rate}, and which is influenced by   the behavior of $\varphi$ (solution of \rife{vhje-sub}) near the boundary $\domeg$ (changing according to the values of $m$) and eventually by the case that  $c=0$. We prove the optimality of such estimates  in star-shaped
domains in  cases when $f+c < 0$ on $\omegb$: in such situations, we show that it can
actually happen that $u(x,t)+ct\to -\infty.$ 

\vskip0.6em
This article is organized as follows. In Section \ref{ergodicproblem}, we
recall the main results about the problem (\ref{vhje-sub}); these
results are mainly borrowed from Lasry and Lions \cite{LL} and
Porretta and V\'eron \cite{PV} where the precise behavior of
$\varphi$ near $\domeg$ is described. Section~\ref{stationaryconvergence} is
devoted to  the easy case when
(\ref{vhje-s})-(\ref{vhjes-bc}) has a solution; in such situation $u(\cdot, t)$ is bounded and converges to the (unique) solution of the stationary problem. We recall that the existence of a stationary solution  corresponds, for the ergodic constant $c$ of \rife{vhje-sub},  to the case $c<0$ (see Proposition \ref{existencecondition}). In Section  \ref{bd-estimates} we consider the case $c\geq 0$ and we state the estimates on the blow-up rate of $u$ proving, as a consequence,  that  
$\frac{u(x,t)}{t} \to
-c$ locally uniformly in $\Omega$, in the whole range $1<m\leq 2$.
In the Section \ref{ergodicconvergence}, we prove the result (i) above. Finally, in Section \ref{nonconvergence} we show that  a similar 
result cannot hold in general in the range $m\geq \frac 32$ or if $c=0$
proving that, in these cases, we have $u+ct \to -\infty$ at least in some  circumstances. We leave to the Appendix the construction of sub and supersolutions, for the problem (\ref{vhje})--(\ref{vhje-id})--(\ref{vhje-bc}), which we use to obtain the crucial estimates on the blow-up rate of $u$ when $c\geq 0$.
\vskip1em
Throughout this paper, we assume that $\Omega$ is a domain of class $C^2$. We denote by $\dd$ the signed-distance to $\domeg$, which is positive in $\Omega$, i.e. $\dd (x):=\inf \{ |x-y| : y \in\partial\Omega\}$ if $x\in \Omega$ and negative in the complementary of $\Omega$ in $\R^N.$ As a consequence of the regularity of $\domeg,$ $\dd$ is a  $C^2$- function in a neighborhood $\mathcal{W}$ of $\domeg.$ We also denote by $\nu$ the $C^1$- function defined by
$\nu(x)=-D \dd (x)\hbox{ in }\mathcal{W};$
if $x\in \domeg,$ then $\nu(x)$ is just the unit outward normal vector to $\domeg$ at $x.$

\section{Preliminary Results on the Stationary Ergodic and Dirichlet Problems}\label{ergodicproblem}

In this section, we deal with  properties of the pair $(c, \varphi)$ solution of (\ref{vhje-sub}). 
We start by recalling the main results concerning $(c,\varphi)$ which were proved in \cite{LL} and  \cite{PV}.
Next, we provide some useful properties of $c$ giving a relationship between the stationary ergodic problem (\ref{vhje-sub}) and the stationary Dirichlet problem (\ref{vhje-s})-(\ref{vhjes-bc}). We end the section with an example to illustrate the possibly unsolvability of (\ref{vhje-s})-(\ref{vhjes-bc}). 

We start with
\begin{theor}[On the Stationary Ergodic Problem]\label{limits}\text{}\\
Assume that $1< m\leq 2$ and that $f\in C(\omegb).$ There exists a unique constant $c \in \R$ such that the problem (\ref{vhje-sub}) has a solution $\varphi \in W_{loc}^{2,p}(\Omega)$ for every  $p>1$. This solution is unique up to an additive constant and satisfies the following properties
\begin{itemize}
\item[(i)] If $1<m<2,$ then 
\begin{equation}\label{expansion4}
\varphi(x) = C^*\dd
^{-\frac{2-m}{m-1}}(x)(1+o(1))\;\hbox{ as }\; \dd
(x)\to 0\;\hbox{ , with  }C^*=\frac{(m-1)^{\frac{m-2}{m-1}}}{(2-m)}\, .
\end{equation}
\item[(ii)] If $m=2,$ then
\begin{equation}\label{expansion4'}
\varphi(x) = |\log \dd
(x)|(1+o(1))\;\hbox{ as }\; \dd
(x)\to 0.
\end{equation}
\item[(iii)] There exists a constant $\Lambda>0$ depending only on $m,$ ${\Omega}$ and $f$ such that 
\begin{equation}\label{gradientbound}
|D \varphi(x)|\leq \Lambda
\dd
^{-\frac{1}{m-1}}(x)\ \quad\hbox{ in } \;\Omega.
\end{equation}
\item[(iv)] We have
\begin{equation}\label{pv}
\lim\limits_{\dd
(x)\to 0} \dd
^{\frac1{m-1}}(x)D \vfi(x)=c_m\,\nu(x)
\end{equation} where $c_m= (m-1)^{-\frac{1}{m-1}}.$
\end{itemize}
\end{theor}
As far as the proof of Theorem~\ref{limits} is concerned, we refer the reader to the proofs of \cite[Theorems I.1, IV.1 and VI.1]{LL} for (i), (ii), (iii) and to \cite[Theorem 1.1-B]{PV} for (iv).
\vskip0.5em
We continue with further estimates on the solutions of (\ref{vhje-sub}), which can be deduced by the previous ones.
\begin{lemma}\label{properties} Assume $1<m\leq 2$, $f\in W^{1,\infty}(\Omega)$ and let $\vfi $ be a solution of (\ref{vhje-sub}) given by Theorem~\ref{limits}. Then $\varphi\in C^{2,\beta}_{loc}(\Omega)$ for all $\beta\in (0,1)$ and

\vskip0.2em
\noindent (i) There exists a  constant $K_0>0$ such that, for every $x\in \Omega,$ we have
\begin{equation}\label{Hessianfi1}
|D^2\vfi(x)|\leq K_0\, \dd^{-\alpha-2}(x)\quad\hbox{ with }\quad \alpha=\frac{2-m}{m-1}\,.
\end{equation}
\vskip0.2em
\noindent(ii) There exists  $K_1>0$ and $\sigma>0$  such that, for every $x\in \Omega$ with $0\leq \dd(x)<\sigma$
\begin{equation}\label{Hessianfi2}
 |D^2\vfi(x)| \leq K_1 \,|D \vfi(x)|^m 
\end{equation} and 
\begin{equation}\label{fi0}
\left\{
\begin{array}{rl}
&|\vfi(x)|\leq  K_1\,  |D \vfi(x)|^{2-m} \quad\quad\quad\quad\hbox{ if } 1<m<2,\\
&\\
&|\vfi(x)|\leq  |\log(|D \vfi(x)|)|+K_1\quad\quad\quad\hbox{ if } m=2.
\end{array}
\right.
\end{equation}
\end{lemma}
\begin{proof}[\bf Proof of Lemma \ref{properties}] We first remark that the $C^{2,\beta}_{loc}$--regularity of the solutions of (\ref{vhje-sub}) comes from the additional assumption on $f$ ($f\in W^{1,\infty}(\Omega)$) and a standard bootstrap argument~: indeed, by Theorem~\ref{limits}, $\varphi$ is in $W_{loc}^{2,p}(\Omega)$ for every  $p>1$, hence in $C^{1,\beta}_{loc}(\Omega)$ for any $\beta \in (0,1)$, and therefore a standard regularity result implies that $\varphi\in C^{2,\beta}_{loc}(\Omega)$ for all $\beta\in (0,1)$ since $|D \varphi|^m$ and $f$ are in $C^{0,\beta}_{loc}(\Omega)$.

In order to prove (i), let $x_0\in \Omega,$ and $r=\frac{\dd (x_0)}2$. We introduce the function $\psi(x)=r^{\al}\vfi(x_0+rx)$ and compute
\begin{equation*}
D\psi(x)=r^{\al+1}D\vfi(x_0+rx) \quad\text{ and }\quad D^2\psi(x)=r^{\al+2}D^2\vfi(x_0+rx).
\end{equation*} 
Since $m(\al+1)=\al+2,$ it is easy to see that $\psi$ is a viscosity solution of
$$ -\Delta \psi+|D \psi|^m=r^{\al+2} (f(x_0+rx)+c)\quad \hbox{ for any } x\in B(0,1)\,. $$ Since
$r^{\alpha+2} f$ and  $|D
f|r^{\alpha+2}$ are in $L^\infty(\Omega),$ we can use the interior estimates  available in \cite[Theorem A.1]{LL} and obtain
$$ \|D \psi\|_{L^\infty(B(0,\frac12))}\leq K\,$$
and then, by elliptic regularity (see e.g. \cite{GT}),
$$ \|D^2 \psi\|_{L^\infty(B(0,\frac12))}\leq K'\, $$
which yields $$|D^2\vfi(x_0+rx)| \leq K'r^{-\al-2} \quad\hbox{ for all } x\in B(0,1/2).$$ By taking $x=0$ and remembering the definition of $r,$ we  obtain (\ref{Hessianfi1}) with $K_0:=2^{\al+2}K'.$ The estimate (\ref{Hessianfi2}) is a consequence of (\ref{pv}) which implies
\begin{equation}\label{pv'}
|D \vfi(x)|\geq K\, \dd^{-\frac1{m-1}}(x)
\end{equation}
for some constant $K>0$ and for $x$ in a suitable neighborhood of $\domeg$. Since $\frac1{m-1}=\frac{\alpha+2}m,$ combining (\ref{pv'})  with (\ref{Hessianfi1}) we obtain 
$$|D^2\vfi(x)|\leq K_0\, \dd^{-\alpha-2}(x)=K_0 \dd^{-\frac{m}{m-1}}(x)\leq K_0\biggl(\frac{1}{K}|D \vfi(x)|\biggr)^m$$
hence   (\ref{Hessianfi2}) for any  $K_1>K_0/K^m$. As far as (\ref{fi0}) is concerned, we use (\ref{expansion4}) and (\ref{pv'}) if $m<2$ and   (\ref{expansion4'}) and (\ref{pv'}) if $m=2$ obtaining
that \rife{fi0} holds true for some $K_1>0$.

\end{proof}

\vskip0.5em

We continue by showing monotonicity and stability properties of the ergodic constant $c$ with respect to the domain.

\begin{prop}\label{ergodicdependence} \item[(i)]
Let $\Omega'$ be an open bounded subset of $\R^N$ such that $\Omega\subseteq\Omega'.$ Let $c_\Omega$ and $c_{\Omega'}$ be the ergodic constants associated to (\ref{vhje-sub}) in $\Omega$ and  $\Omega'$ respectively.
Then we have
\begin{equation}\label{ergodicestimates}
c_\Omega\leq c_{\Omega'}.
\end{equation} 
\item[(ii)] Let $\Omega'$ be an open bounded subset of $\R^N$ such that $\Omega\subset\subset\Omega'.$ The respective ergodic constants $c_\Omega$ and $c_{\Omega'}$ of $\Omega$ and $\Omega'$ satisfy  
\begin{equation}\label{ergodicestimates1}
c_\Omega< c_{\Omega'}.
\end{equation} 
\item[(iii)] Moreover, the ergodic constant $c$ depends continuously on $\Omega.$ Otherwise said, for $0<\eta<1,$ if $c_\eta$ is the ergodic constant in (\ref{vhje-sub}) set in $\Omega+B(0,\eta)$\footnote{$\Omega+B(0,\eta):=\{x\in \R^N: \dd (x) > - \eta\}.$}
then  
\begin{equation}\label{dependence}
c_\eta\downarrow c_\Omega \quad\hbox{ as }\quad \eta\downarrow0
\end{equation} where $B(0,\eta)$ is a ball 
of radius  $\eta.$
\end{prop} 
\vskip0.5em
\begin{proof}[\bf Proof of Proposition~\ref{ergodicdependence}] (i) Let $(c_{\Omega},\varphi_{\Omega})$ and $(c_{\Omega'},\varphi_{\Omega'})$ be the pair of  solutions of the ergodic problem (\ref{vhje-sub})  in $\Omega$ and $\Omega'$ respectively. From what we obtained above, the constants $c_{\Omega}$ and $c_{\Omega'}$ are unique whereas the functions $\varphi_{\Omega}$ and $\varphi_{\Omega'}$ are unique up to a constant.
We study $\max_{\omegb}(\mu\varphi_{\Omega'}-\varphi_{\Omega})$ for some $\mu\in(0,1)$ close to $1$.  Now, we need to show that  the maximum is achieved inside $\Omega$.

Observe that, since $\Omega\subseteq\Omega',$ it could happen that  $\domeg\cap\partial\Omega'\neq\emptyset,$ meaning that the two domains $\Omega$ and $\Omega'$ touch at some points. Due to the behavior of  $\vfi_{\Omega'}$ and $\vfi_{\Omega}$ near $\partial\Omega'$ and $\partial\Omega$ respectively, we first deal with points on $\partial\Omega$ which do not belong to $\partial\Omega'$ and next we treat common points  of $\partial\Omega$ and $\partial\Omega'.$

We pick any $x_0\in \domeg$  such that $x_0\notin\partial\Omega';$ in this case, since  $\vfi_{\Omega'}$ is bounded in $\omegb\setminus(\domeg\cap\partial\Omega'),$ it  follows that
\begin{equation*}
\lim\limits_{x\to x_0} \, (\mu\varphi_{\Omega'}-\varphi_{\Omega})(x)= \mu\varphi_{\Omega'}(x_0)- \lim\limits_{x\to x_0} \varphi_{\Omega}(x)=-\infty.
\end{equation*}
On the other hand, if   $x_0\in\domeg\cap\partial\Omega',$ using the asymptotic behavior  \rife{expansion4} and that, if $d_{\partial \Omega'}$ denotes the distance to $\partial \Omega'$, $d_{\partial \Omega'}(x)\geq \dd (x)$ we have, when $1<m<2$
\begin{equation*}
(\mu\varphi_{\Omega'}-\varphi_{\Omega})(x)\leq   C^*(\mu-1) \dd
^{-\frac{2-m}{m-1}}(x) (1+o(1)) 
\quad\hbox{ as } x\to x_0 \,,
\end{equation*} 
hence $(\mu\varphi_{\Omega'}-\varphi_{\Omega})(x)\to -\infty$ in this case too. When $m=2$, the same conclusion holds   by the use of (\ref{expansion4'}).
Therefore, in any case  it follows that $(\mu\varphi_{\Omega'}-\varphi_{\Omega})$ has a maximum point $x_\mu\in \Omega$.   
Going back to the equations solved by $\mu\varphi_{\Omega'}$ and $\varphi_{\Omega},$ we obtain:
\begin{eqnarray}\label{eq1}
-\Delta (\mu\vfi_{\Omega'})(x_\mu)+|D(\mu\vfi_{\Omega'})(x_\mu)|^m&=&-\mu\Delta \vfi_{\Omega'}(x_\mu)+\mu^m|D\vfi_{\Omega'}(x_\mu)|^m\nonumber\\
&\leq&\mu(-\Delta \vfi_{\Omega'}(x_\mu)+|D\vfi_{\Omega'}(x_\mu)|^m)\nonumber\\&=& \mu f(x_\mu)+\mu c_{\Omega'}
\end{eqnarray} and
\begin{equation}\label{eq2}
\quad -\Delta \vfi_\Omega(x_\mu)+|D\vfi_\Omega(x_\mu)|^m=f(x_\mu)+c_\Omega.
\end{equation}
By subtracting (\ref{eq2}) from (\ref{eq1}) and using the following properties: $D(\mu\vfi_{\Omega'})(x_\mu)=D\vfi_{\Omega}(x_\mu)$ and $\Delta (\mu\vfi_{\Omega'}-\vfi_{\Omega})(x_\mu)\leq0,$ one gets:
$c_{\Omega}\leq \mu c_{\Omega'}+(\mu-1) f(x_\mu)$ and (\ref{ergodicestimates}) follows by sending $\mu$ to $1.$
\vskip0.3em
\noindent (ii) Let $\Omega'$ be such that $\Omega\subset\subset\Omega'$; from (\ref{ergodicestimates}) it follows that $c_{\Omega}\leq c_{\Omega'}.$ Moreover, we would like to show that for this case, we have: $c_{\Omega}< c_{\Omega'};$ to do so, we assume the contrary by setting $c_{\Omega}=c_{\Omega'}.$ Using the boundedness of $\varphi_{\Omega'}$ on $\omegb$ and behavior of $\varphi_{\Omega}$ near $\domeg,$ it is easy to see that $\varphi_{\Omega'}-\varphi_{\Omega}\to-\infty$ on $\domeg,$ meaning that $\varphi_{\Omega'}-\varphi_{\Omega}$ achieves its global maximum on $\omegb$ at some $\xb$ inside $ \Omega.$ On the other hand,  the convexity of $p \mapsto |p|^m$ yields $$ |D \varphi_{\Omega'}|^m \geq  |D \varphi_{\Omega}|^m + m |D \varphi_{\Omega}|^{m-2}D \varphi_{\Omega}\cdot (D (\varphi_{\Omega'}-\varphi_{\Omega})).$$
With this argument and using the local bounds on $|D\varphi_{\Omega} |$ in (\ref{gradientbound}), we have that $\varphi_{\Omega'}-\varphi_{\Omega}$ solves \begin{equation}\label{linearization}
-\Delta(\varphi_{\Omega'}-\varphi_{\Omega})-C(x) |D(\varphi_{\Omega'}-\varphi_{\Omega})|\leq0\ \hbox{ in }\Omega
\end{equation} for some  $C(x)>0$ which is   bounded in any compact subset.  Applying  the Strong Maximum Principle (see \cite[Lemma 2.1]{TTT}) we find that $\varphi_{\Omega'}-\varphi_{\Omega}$ is constant in $\omegb$ which clearly leads to a contradiction since $\varphi_{\Omega'}-\varphi_{\Omega}$ blows up on the boundary $\domeg,$ whence (\ref{ergodicestimates1}) holds.

\vskip0.3em
\noindent(iii) Now, we turn to the proof of the continuous dependence of $c$ and $\vfi$ in $\Omega,$ namely (\ref{dependence}). Let $0<\eta<\eta'<1,$ since $\Omega+B(0,\eta)$ is a strict subset of $\Omega+B(0,\eta')$,  we find from (\ref{ergodicestimates1}) that $c_{\eta}< c_{\eta'},$ meaning that the sequence $(c_\eta)_{0<\eta<1}$ decreases as $\eta$ goes to $0$. On the other hand, knowing that $\Omega\subset\Omega+B(0,\eta)$ for any $0<\eta<1,$ again by (\ref{ergodicestimates}), we obtain $c_\Omega<c_\eta.$ Therefore, $(c_\eta)_{0<\eta<1}$ is convergent in $\R$ and $c_\eta\downarrow c_0=c_\Omega $ as $\eta\to0^+$ by using the arguments on stability of (\ref{vhje-sub}) and the uniqueness of the ergodic constant.
\end{proof}
\vskip0.5em
Hereafter, to stress on the dependence of the initial boundary value problem (\ref{vhje})-(\ref{vhje-id})-(\ref{vhje-bc}) on $f$, $u_0$ and  $g$,  we denote it by $E(\Omega,f,g,u_0)$ where \lq\lq$E$\rq\rq\ stands  for Evolution. Likewise, we denote by $S(\Omega,f,g),$ the boundary value problem (\ref{vhje-s})-(\ref{vhjes-bc})  where \lq\lq $S$\rq\rq\ stands for Stationary.

Now, we link the problems (\ref{vhje-sub}) and $S(\Omega,f,g)$ by pointing out that the existence of a solution for $S(\Omega,f,g)$ depends on the ergodic constant $c$ in (\ref{vhje-sub}).
\begin{prop}\label{existencecondition} Let $c$ be the ergodic constant associated to (\ref{vhje-sub}) and let us denote by $\mathcal{S}$ the set of all $\la\in \R$ such that there exists a viscosity subsolution $\phi\in C(\overline\Omega)$ of
\begin{equation}\label{additive-equation}
-\Delta \phi+|D \phi|^m \leq  f+\la \quad\hbox{ in } \Omega.
\end{equation}  Then
\begin{equation}\label{const-ergodic}
c=\inf\ \{\la:\la\in\mathcal{S}\}\;
\end{equation} and the infimum in (\ref{const-ergodic}) is not attained.
Moreover, a necessary and sufficient condition for $S(\Omega,f,g)$ to have a viscosity solution is that $c<0.$
\end{prop}  
\begin{remark} It is worth mentioning that by a viscosity solution of $S(\Omega,f,g),$ we mean a $C(\omegb)$- function satisfying (\ref{vhje-s}) in the viscosity sense and (\ref{vhjes-bc}) pointwisely. Indeed, when $1<m\leq 2,$ there is no loss of boundary conditions as specified in \cite[Propositions 3.1 and 3.2]{BDL}, and   $S(\Omega,f,g)$ is   a classical Dirichlet problem.

\end{remark}
\begin{proof}[\bf Proof of Proposition \ref{existencecondition}]  In order to prove (\ref{const-ergodic}), we first remark that $\mathcal{S}\neq \emptyset;$ indeed, it is easy to see that $\|f\|_\infty\in \mathcal{S}$ because $\varphi\equiv 0$ is a subsolution of 
$$  -\Delta \varphi +|D \varphi |^m \leq  f+\|f\|_\infty\quad\hbox{ in }\Omega.
$$ 
On the other hand, we pick any $\la\in \cal{S}$ and denote by $\phi_\la\in C(\omegb),$ the function satisfying $-\Delta \phi_\la+|D \phi_\la|^m \leq  f+\la$ in the viscosity sense. We argue as in the proof of Proposition~\ref{ergodicdependence}, by studying $\max_{\omegb}(\phi_{\la}-\varphi)$, in order to reach the conclusion that $c\leq\la.$ Therefore, the right hand side of (\ref{const-ergodic}), which we denote by $\la^*,$ is well-defined and $c\leq \la^*.$

Now, we assume that $c< \la^*,$ by keeping the notations of Proposition~\ref{ergodicdependence}-(iii), there exists some $0<\eta_0\ll1$ such that, for any $0<\eta\leq \eta_0,$ we have $c_\eta\in ]c,\frac{c+\la^*}{2}[.$ But, if $\varphi_{\eta}$ solves
 $$-\Delta \varphi_{\eta}+|D \varphi_{\eta}|^m =  f+c_\eta\quad\hbox{ in }\Omega_\eta:=\Omega+B(0,\eta)$$ then $\varphi_{\eta}$ is in $W^{2,p}_{loc} (\Omega_\eta)$ for any $p>1$ and therefore $\varphi_{\eta}\in C(\omegb)$; thus $\varphi_{\eta}$ is a viscosity subsolution of the equation associated to $c_\eta$ in $\Omega$ and $c_\eta\in \cal{S}.$ This clearly is a contradiction since $c_\eta<\la^*$ and we conclude that $c=\la^*$ and then (\ref{const-ergodic}) holds.

Moreover, the infimum in (\ref{const-ergodic}) is not attained: indeed, if $\varphi_c \in C(\omegb)$ is a subsolution of the $c$-equation (\ref{additive-equation}) and if $\vfi$ is a solution of the ergodic problem (\ref{vhje-sub}), then the $\max_{\omegb}\,(\varphi_c - \varphi)$ is achieved at some point of $\Omega$ since $\vfi$ blows up on $\domeg$ and applying the Strong Maximum Principle exactly as in the proof of Proposition~\ref{ergodicdependence} (ii), we find that $\varphi_c - \varphi$ is constant in $\Omega$, a contradiction since $\varphi_c \in C(\omegb)$ and $\vfi(x) \to +\infty$ as $x \to \domeg$.

Now we turn to the proof of the second part of the result. If the generalized Dirichlet problem $S(\Omega,f,g)$ has a  bounded viscosity solution, then $0 \in \mathcal{S}$, and then the first part implies that necessarily $c<0$.

Conversely, we assume that $c<0$ and prove that $S(\Omega,f,g)$ has a unique bounded viscosity solution. To do so, we are going to apply the Perron's method (cf. \cite{Ishii}, \cite{CIL} and \cite{DaLio0}) and in order to do it, we have to build sub and supersolution for $S(\Omega,f,g)$.

For the subsolution, since $c<0,$ we find from (\ref{dependence}) (see Proposition~\ref{ergodicdependence}), the existence of $0<\eta_0<1$ such that $c_\eta\leq0$ for all $\eta\leq\eta_0$ where $c_\eta$ is such that there exists a function $\varphi_\eta$ which is a viscosity solution of the ergodic problem (\ref{vhje-sub}) in $\Omega+B(0,\eta)$ for all $\eta\leq \eta_0.$ 

Since $g$ is bounded on $\domeg$ whereas $\vfi_\eta$ is bounded in $\omegb,$ there exists $K>0$ such that $\vfi_\eta-K \leq g$ on $\domeg$. Therefore we have, at the same time, a subsolution of $S(\Omega,f,g)$ required in the Perron's method, but also a strict subsolution of $S(\Omega,f,g)$ which implies a comparison result for $S(\Omega,f,g)$ (See \cite[Theorem 2.3]{TTT}).

For the supersolution, it is easy to see that $l(x)=|x-x_{0}|^{2}+ \|g\|_{\infty}+1$ is a supersolution of $S(\Omega,f,g)$ for some $x_0\in \mathbb{R}^N$ such that $B(x_{0},(\|f\|_{\infty}+2N)^{1/m})\cap \overline{\Omega}=\emptyset.$ Using the (strict) subsolution $\vfi_\eta-K$ built above, the existence of the solution therefore follows by combining the comparison result and the classical Perron's method.
\end{proof}

\vskip0.5em

We end this section by giving an example showing that the stationary Dirichlet Problem $S(\Omega,f,g)$ does not always have a solution. 
\begin{example}\label{nonsolvable2}\rm We fix some $R>0$ and consider the one-dimensional equation  
\begin{equation}\label{example}
-\eta''+|\eta'|^m=-C^m \ \ \ \hbox{ in } \ \ \ (-R,R)
\end{equation} with $C>0.$ If $\eta$ solves (\ref{example}), then after some easy change of variable and computations, we find that 
$$\dfrac{1}{C^{m-1}}\int_{\frac{\eta'(0)}{C}}^{\frac{\eta'(x)}{C}}\dfrac{ds}{|s|^m+1}=x.$$ It follows that
$$C^{m-1}x=\int_{\frac{\eta'(0)}{C}}^{\frac{\eta'(x)}{C}}
\dfrac{ds}{|s|^m+1}\leq\int_{-\infty}^{+\infty}
\dfrac{ds}{|s|^m+1}:= K.$$ Therefore, letting $x\rightarrow R,$ we obtain :
\begin{equation}\label{example*}
C^{m-1}\leq \dfrac{K}{R}
\end{equation}
Since  $R>0$ is fixed and $m>1,$ by choosing $C>0$ large enough, we find that the inequality (\ref{example*}) cannot hold and we conclude that the ordinary differential equation (\ref{example}) is not solvable for large $C.$ Therefore, (\ref{vhje-s})-(\ref{vhjes-bc}) is not always solvable as specified above when we considered $f:=-C^m<0.$ \hfill$\Box$
\end{example}
\vskip0.5em

\section{Asymptotic Behavior for the Parabolic Problem}\label{conv}
This section is devoted to the   description of the asymptotic behavior  of the  solution $u$ of the initial boundary-value problem $E(\Omega,f,g,u_0)$.

\subsection{Convergence to the solution of the  Stationary Dirichlet Problem}\label{stationaryconvergence}

We start with  the case that $u$ converges to the solution of the stationary Dirichlet problem. The main result of this subsection is the following.

\begin{theor}\label{convergence1} Let $1<m\leq 2,$ $f\in C(\overline{\Omega}),$ $u_0\in C(\overline{\Omega})$ and $g\in C(\partial\Omega).$ 
Assume that $S(\Omega,f,g)$ has a unique viscosity solution which we denote by $u_\infty.$ Let $u$ be the unique continuous viscosity solution of $E(\Omega,f,g,u_0).$ Then, as  $t\to+\infty,$
\begin{equation}\label{stationarybehavior}
u(x,t)\to u_\infty(x) \quad\hbox{ uniformly for all } x\in \overline{\Omega}.  
\end{equation}
\end{theor}
\vskip0.5em
\begin{proof}[\bf Proof of Theorem \ref{convergence1}] 
We first notice that $u$ is uniformly bounded on $\overline{\Omega}\times[0,+\infty).$ Indeed, since $u_\infty$ solves $S(\Omega,f,g),$ it is 
straighforward that $u_\infty$ solves $E(\Omega,f,g,u_\infty)$ and by \cite[Corollary 2.1]{TTT}, one gets
$$\|u(x,t)-u_\infty\|_\infty\leq\|u_0-u_\infty\|_\infty .$$
Next, from the uniform boundedness of $u$ obtained above, we use the half-relaxed limits method to say that the functions $$\overline{u}(x)=\underset{\underset{t\to +\infty}{y\to x}}{\limsup}\ u(y,t)\quad\hbox{ and }\quad\underline{u}(x)=\underset{\underset{t\to +\infty}{y\to x}}{\liminf}\ u(y,t)$$ 
are respectively subsolution and supersolution of
(\ref{vhje-s})-(\ref{vhjes-bc}). By definition of the half limits, we have
$\underline{u}\leq\overline{u}$ on $\Omega$  but given that the ergodic constant $c$ is strictly negative (see Proposition \ref{existencecondition}), there exists a strict subsolution for $S(\Omega,f,g)$ and we can therefore apply \cite[Theorem 2.3]{TTT} to obtain $\overline{u}\leq\underline{u}$ on
$\omegb.$ It is worth noticing that since no loss of boundary condition could happen (see \cite[Propositions 3.1 and 3.2]{BDL}), the Dirichlet  condition (\ref{vhjes-bc}) is understood in the classical sense: $\underline{u}\leq g$ on $\partial\Omega$ and $\overline{u}\geq g$ on $\partial\Omega.$ 
Thus, we obtain $u_\infty=\underline{u}=\overline{u}$ on $\overline{\Omega},$ meaning that (\ref{stationarybehavior}) holds.

\end{proof} 

\vskip1em
\subsection{Convergence of $\ \dfrac{u(x,t)}{t}$}\label{bd-estimates}
In this subsection, we assume that $c\geq 0.$ Indeed, due to Proposition \ref{existencecondition}, the case $c<0$ is already described by Theorem \ref{convergence1}.  The main result of this section is the following convergence of $\frac{u(x,t)}{t}$ to $-c$ which always holds true, even if at different rates.

\begin{theor}\label{rate}
Let $1<m\leq 2$, $f\in W^{1,\infty}(\Omega),$ $u_0\in C(\overline{\Omega})$ and $g\in C(\partial\Omega).$  Let $u$ be the unique continuous viscosity solution of $E(\Omega,f,g,u_0).$ Then we have
$$
\frac{u(x,t)}t\to -c \quad\hbox{ locally uniformly on }  \Omega \quad\hbox{ as } t \to +\infty\,.
$$
In particular, for any compact set $K\subset \Omega$ there exists a constant $M_K$ such that, as  $t\to+\infty$,
\begin{itemize}

\item[(i)] if  $c>0$ then
\be\label{rate1}
\begin{cases}
\|\frac{u(x,t)}t + c\|_{C(K)}\leq \frac{M_K}t\,  &  \hbox{when $\frac32 <m\leq 2$}
\\
\m
\|\frac{u(x,t)}t + c\|_{C(K)}\leq M_K\, \frac{\log t}{t  }& \hbox{when $m=\frac32$}
\\
\m
\|\frac{u(x,t)}t + c\|_{C(K)}\leq \frac{M_K}{t^{\frac{m-1}{2-m} } }& \hbox{when $1<m< \frac32$}
\end{cases}
\ee

\item[(ii)] if $c=0$ then
\be\label{rate2}
\begin{cases}
\|\frac{u(x,t)}t\|_{C(K)}\leq M_K\, \frac{\log t}t\,  &  \hbox{when $m= 2$}
\\
\m
\|\frac{u(x,t)}t\|_{C(K)}\leq \frac{M_K}{t^{2-m } }& \hbox{when $1<m<  2$}.
\end{cases}\ee
\end{itemize}
\end{theor}

\vskip0.5em
Theorem~\ref{rate} is a direct consequence of the following estimate. For technical reasons, it will be convenient   to consider  the unique solution $\vfi_0$ of (\ref{vhje-sub}) satisfying $\min_{\omegb}\vfi_0=0.$ In this way, as said above, any solution $\vfi$ of (\ref{vhje-sub}) is described as $\vfi=\vfi_0+k$ for some constant $k.$ 
\vskip0.5em
\begin{theor}\label{supersolution}
Let $\Omega$ be a domain of class $C^2$. Let $f\in W^{1,\infty}(\Omega)$.
Then we have, as $t\to + \infty$
\vskip0.5em

(i) If $c>0$   there exists a constant $M>0$ such that
\begin{equation}\label{tesi2}
\begin{cases}
u+ct\geq \ga(t)\vfi_0(x-\mu(t)n(x))- M& \hbox{if $\frac32<m\leq 2,$}
\\
u+ct\geq \ga(t)\vfi_0(x-\mu(t)n(x))- M\,\log t& \hbox{if $m=\frac32,$}
\\
u+ct\geq \ga(t)\vfi_0(x-\mu(t)n(x))- M \,t^{\frac{3-2m}{2-m}}& \hbox{if $1<m<\frac32,$}
\end{cases}
\end{equation}
where $n(x)$ is  a vector field such that $n(x)\cdot D d(x)<0$ and $\ga(t)$, $\mu(t)$ are positive continuous functions such that  $\ga(t)$  is increasing and $\ga(t)\uparrow 1$, $\mu(t)$ is decreasing and $\mu(t)\downarrow 0$ as $t\to \infty$. 

\vskip0.5em

(ii) If $c=0$  there exists a constant $M>0$ such that
\be\label{tesi3}
\begin{cases}
u\geq \ga(t)\vfi_0(x-\mu(t)n(x))- M\log t& \hbox{if $m=2$,}
\\
u\geq \ga(t)\vfi_0(x-\mu(t)n(x))- M \,t^{2-m}& \hbox{if $1<m<2$,}
\end{cases}
\ee
where $\ga(t)\uparrow 1$, $\mu(t)\downarrow 0$ as $t\to +\infty$.
\end{theor}

Knowing that $\varphi_0$ is nonnegative and solves (\ref{vhje-sub}), it is easy to observe that $\vfi_0+\|u_0\|_{L^\infty(\Omega)}-ct$  is always a super-solution  of (\ref{vhje})-(\ref{vhje-id})-(\ref{vhje-bc}). Therefore, by means of Strong Comparison Principle for $E(\Omega,f,g,u_0)$ (see \cite[Theorem 3.1]{BDL} or \cite[Theorem 2.1]{TTT}), we have for all $1<m\leq 2,$
\begin{equation}\label{abovebounds}
 u(x,t)+ct\leq \vfi_0(x)+\|u_0\|_{L^\infty(\Omega)} \quad\hbox{ in } \Omega\times [0,+\infty). 
\end{equation}
Theorem~\ref{supersolution} provides some estimates on $u(x,t)+ct$ when $c\geq0$ which can be  used to locally bound $u+ct$ from below in order to complement (\ref{abovebounds}).
  \vskip0.5em
We refer the reader to the Appendix for the proof of Theorem~\ref{supersolution} and turn to the

\vskip0.5em
\begin{proof}[\bf Proof of Theorem~\ref{rate}]
It follows from \rife{abovebounds} that, in any compact set $K$, we have  $\frac{u(x,t)}t+c\leq \frac {M_K}t$ for some constant $M_K$. On the other hand the estimate from below varies  according to the values of $m$ and whether $c=0$ or $c>0.$ Using the estimates \rife{tesi2} and \rife{tesi3} in Theorem \ref{supersolution} and taking into account that $\vfi_0$ is locally bounded, we immediately deduce \rife{rate1} and \rife{rate2} and in particular that $\frac{u(x,t)}t\to -c$ locally uniformly as $t\to +\infty$.

\end{proof}

In the case that $c>0$ and $1<m\leq \frac32$, or if $c=0$,   the rates of convergence given above cannot in general be improved, as we will see later (Theorem~\ref{subsolution}). 
\vskip1.0em

\subsection{Convergence to the Stationary Ergodic Problem when $\frac{3}{2}<m \leq 2$ and $c>0$}\label{ergodicconvergence}
The goal of this section is to describe the asymptotic behavior of the solution $u$ of the generalized initial boundary-value problem
(\ref{vhje})-(\ref{vhje-id})-(\ref{vhje-bc}) in connection with the stationary ergodic problem (\ref{vhje-sub}). 
Our main result is the following: 
\begin{theor}[Convergence result]\label{timeconvergence}\text{}\\
Let $f\in W^{1,\infty}(\Omega),$ $u_0\in C(\overline{\Omega})$ and $g\in C(\partial\Omega).$ Let $c \in \mathbb{R}$ and  $\varphi_0\in W_{loc}^{1,\infty}(\Omega)$ be the unique viscosity solution of the ergodic problem (\ref{vhje-sub}) such that $\ \min_\Omega \vfi_0=0.$ Let $u$ be the unique continuous viscosity solution of $E(\Omega,f,g,u_0).$ Assume that $c>0$ and $\frac{3}{2}<m\leq  2,$ then we have 
\begin{equation}\label{ergodic}
u(x,t)+ct\to \varphi_0(x) + C\quad\hbox{ locally uniformly in }  \Omega, \;\hbox{ as } t \to +\infty
\end{equation} for some constant $C$ depending on $\Omega,$ $f,$ $c,$ $u_0,$ and $g.$
\end{theor} 
Recalling that all solutions of problem (\ref{vhje-sub}) only differ by addition of a  constant, one can rephrase (\ref{ergodic}) saying that $u(x,t)+ct$ converges to  a solution of the ergodic problem  (\ref{vhje-sub}). We choose to represent all solutions as $\vfi_0+C$  in order to emphasize that there is precisely one constant $C$, depending on the data, which determines  the asymptotic limit of $u+ct-\vfi_0$. 

\vskip0.5em
We also notice that by combining (\ref{abovebounds}) and (\ref{tesi2}), it follows that
\begin{equation}\label{uniformbounds}
\gamma(t)\vfi_0(x-\mu(t)n(x))- M\leq u(x,t)+ct\leq \vfi_0(x)+\|u_0\|_{L^\infty(\Omega)}\quad\hbox{ in } \Omega\times [0,+\infty).
\end{equation} If $\lim\limits_{t\to+\infty}(u(x,t)+ct)$ exists, then by sending $t\to+\infty$ in (\ref{uniformbounds}), we obtain 
$$ \vfi_0(x)- M\leq \lim\limits_{t\to+\infty}(u(x,t)+ct)\leq \vfi_0(x)+\|u_0\|_{L^\infty(\Omega)}$$
{From} the blow-up behavior of $\varphi_0$ near the boundary, it obviously follows that 
\begin{equation}\label{blow-up}
\lim\limits_{t\to+\infty}(u(x,t)+ct)\to+\infty\quad\hbox{ as } x\to\domeg
\end{equation} which gives the expected behavior near the boundary since there is no loss of boundary condition for all $1<m\leq 2.$  

\vskip1.0em
The rest of this section is devoted to the proof of Theorem \ref{timeconvergence}. 
\vskip1.0em

Hereafter, for all $(x,t)\in \omegb\times[0,+\infty),$ we set 
\begin{equation*} 
v(x,t):=u(x,t)+ct \quad\hbox{ and }\quad w(x,t)=v(x,t)-\vfi_0(x).
\end{equation*}
The function $v(\cdot,\cdot + t)$ solves $E(\Omega,f+c,g+ct ,v(\cdot,t))$ whereas $\vfi_0$ is a supersolution of $E(\Omega,f+c,g+ct ,\vfi_0)$ since $\vfi_0$ solves (\ref{vhje-sub}). From the comparison principle, one gets, for all $x\in\overline{\Omega}$ and $s\geq t\geq 0,$ $$\max_{x\in\overline{\Omega}}w(x,s)\leq \max_{x\in\overline{\Omega}}w(x,t).$$ 
It follows that the function $t\mapsto m(t):=\max_{x\in\overline{\Omega}}w(x,t)$ is non-increasing. Moreover, as a by-product of (\ref{uniformbounds}), $m(t)$ is bounded. Therefore $m(t)\downarrow\overline{m}$ as $t\to+\infty.$ 

Now, since we want to deal with bounded functions and clearly $w(x,t) \to -\infty$ when $x \to \domeg$, we choose any constant $K>|\overline{m}|$ and set $$z(x,t):=\sup[w(x,t),-K].$$
We notice that, since $m(t)$ is bounded from below by $\overline{m}$, we still have 
$$m(t):=\max_{x\in\overline{\Omega}}w(x,t)=\max_{x\in\overline{\Omega}}[\sup(w(x,t),-K)]=\max_{x\in\overline{\Omega}}z(x,t)
\ \ \hbox{ for all }\ \  t\geq 0,$$ and from (\ref{uniformbounds}), $z$ is uniformly bounded on $\omegb\times[0,+\infty)$. Moreover, since $w(x,t)\to-\infty$ as $x\to\domeg$ for all $t\geq0,$ we have $z=-K$ on $\domeg\times[0,+\infty)$ and we also remark that $z$ is a viscosity solution of 
\begin{equation}\label{halfrelaxed3}
\left\{
\begin{array}{rl}
\phi_t(x,t)-\Delta \phi(x,t)+\mathcal{H}(x)\cdot D\phi(x,t)&\leq 0  \quad\quad\hbox{ in } \Omega\times(0,+\infty)\\
\phi(x,t)+K&=0  \quad\quad\hbox{ on } \domeg\times(0,+\infty)
\end{array}
\right. 
\end{equation} where
\be\label{calH}
\qquad \mathcal{H}(x)=m|D\vfi_0(x)|^{m-2}D\vfi_0(x).
\ee
Indeed, by the convexity of $p\mapsto |p|^m$, $w$ is a subsolution of the above equation and we recall that the maximum of two subsolutions is a subsolution.

In order to have an equation with continuous coefficients, we introduce the operator
\be\label{L-op}
\mathcal{L}\phi(x,t):=\dd(x)\phi_t(x,t)-\dd(x)\Delta \phi(x,t)+\dd(x)\mathcal{H}(x)\cdot D\phi(x,t)\, ,
\ee 
just obtained from the previous equation by multiplying by $\dd (x)$. From (\ref{pv}), it is easy to note that 
\be\label{calH'}
|D\vfi_0(x)|^{m-2}D\vfi_0(x)= -(m-1)^{-1}\frac{D\dd(x)}{\dd(x)}+o\left(\frac1{\dd(x)}\right)\quad\hbox{ as }\quad \dd(x)\to0.
\ee It is therefore obvious to see that $\cal{H}$ has a singularity on $\domeg$ whereas $\dd(x)\mathcal{H}(x)$ can be extended as a continuous function on $\omegb.$ Using this new operator, we have 
$$\mathcal{L} z \leq 0  \quad\quad\hbox{ in } \Omega\times(0,+\infty)\; ,$$
which replaces the subsolution property in (\ref{halfrelaxed3}).

To complete the proof of (\ref{ergodic}), we first give the following local H\"older continuity of the unique solution $v$ of $E(\Omega,f+c,g+ct,u_0),$ with respect to its $x$ and $t$ variables.
\begin{prop}\label{localholder}Let $1<m\leq  2$ and $v$ be the unique continuous viscosity solution of  $E(\Omega,f+c,g+ct,u_0).$ Then
\begin{itemize}
\item[(i)]  For all $\eta>0$ and for all $\nu\in (0,1),$ we have $v(\cdot,t)\in C^{0,\nu}_{loc}(\Omega)$ for all $t\geq \eta.$ Moreover, if $K_\delta : = \{y\in \Omega\, :\ \dd(y)\geq \delta\}$ for any $\delta>0$, then, for all $t\geq \eta,$ the $C^{0,\nu}$-norm of $v(\cdot,t)$ on $K_\delta$ depends only on $\nu$, $\delta$, $\eta$, $\|f\|_\infty$  and the $L^\infty$-norm of $v$ on $K_{\delta/2} \times[\frac{\eta}2,+\infty)$. 
\item[(ii)] For any $x\in\Omega,$ $v(x,\cdot)\in C^{0,\frac{\nu}{2}}_{loc}(0,+\infty)$ for any $\eta >0$. Moreover the $C^{0,\frac{\nu}{2}}$-norm of $v(x,\cdot)$ in $[\eta,+\infty)$ depends only on $\nu,$ $\eta$, $\|f\|_\infty$ and the $L^\infty$-norm of $v$ on $K_{\dd(x)/2} \times[\eta/2,+\infty)$. 
\end{itemize}
\end{prop} We postpone the proof of Proposition \ref{localholder}-(i) to the Appendix and refer the reader to \cite[Lemma 9.1]{BBL} for the proof of Proposition \ref{localholder}-(ii).

\vskip0.5em  

In order to prove (\ref{ergodic}), we will use the following  
\begin{lemma}\label{global-limits} $z(x,t)\to \overline{m}$ locally uniformly in $\Omega$ as $t\to+\infty.$
\end{lemma} Indeed, from Lemma~\ref{global-limits}, since $K>|\overline{m}|,$ it  easily follows that $w(x,t)\to \overline{m}$ locally uniformly in $\Omega$ as $t\to+\infty$, and we get (\ref{ergodic}) with $C=\overline m$. 
\vskip0.5em 
\begin{proof}[Proof of Lemma~\ref{global-limits}] We split it into several parts.\\
1. Let $\hat{x}\in \Omega$ and $r>0$ such that $\overline{B}(\hat{x},r)\subset\Omega.$ From Proposition~\ref{localholder}, it follows that there exists a sequence $t_n\to +\infty$ as $n\to+\infty$ and a function $\hat z$  such that
$$z(x,t_n)\to\hat{z}(x)\quad\hbox{ uniformly on } \overline{B}(\hat{x},r).$$  
We define $$z_n(x,t):=z(x,t+t_n)\quad \hbox{ for all }(x,t)\in \omegb\times(-t_n,+\infty)$$ and notice that $z_n$ is a viscosity solution of (\ref{halfrelaxed3}) on $\omegb\times(-t_n,+\infty).$ From the uniform boundedness of $z,$ we derive the one of $z_n$ and the half-relaxed 
limits method implies that $$\tilde{z}(x,t)=\underset{\underset{n\to +\infty} {(y,s)\to (x,t)}}\limsup\ z_n(y,s)$$ is a viscosity subsolution of the generalized Dirichlet problem
\begin{equation}\label{halfrelaxed4}
\left\{
\begin{array}{rl}
\mathcal{L}\tilde{z}(x,t)&\leq 0  \quad\quad\hbox{ in } \Omega\times(-\infty,+\infty)\\
\min\{\mathcal{L}\tilde{z}(x,t), \tilde{z}(x,t)+K\}&\leq 0  \quad\quad\hbox{ on } \domeg\times(-\infty,+\infty)
\end{array}
\right. 
\end{equation}

\vskip0.5em 
\noindent 2. We claim that $\underset{x\in\omegb}\max\ \tilde{z}(x,t)=\overline{m}.$ Indeed, on one hand, for all $n\in \mathbb{N},$ there exists $x_n\in \omegb$ such that $$z(x_n,t+t_n)=\underset{x\in\omegb}\max\ z(x,t+t_n)=\underset{x\in\omegb}\max\ z_n(x,t)\; .$$ 
From the compactness of $\omegb,$ we have,  up to subsequence,  $x_n\to\tilde{x}$ for some $\tilde x\in \omegb$. It follows that\begin{eqnarray}
\overline{m}=\underset{n\to+\infty}\lim m(t+t_n)&=&\underset{n\to+\infty}\lim [\underset{x\in\omegb}\max\ z(x,t+t_n)]\nonumber\\
&=&\underset{n\to+\infty}\lim z_n(x_n,t)\,\nonumber
\\
&\leq & \tilde{z}(\tilde{x},t) \leq \underset{x\in\omegb}\max\ \tilde{z}(x,t).\nonumber
\end{eqnarray} On the other hand, for all $x \in \omegb$, we have
\be
\tilde{z}(x,t)=\underset{\underset{n\to +\infty} {(y,s)\to (x,t)}}\limsup z_n(y,s)\leq\underset{(s,n)\to (t,+\infty)}\limsup m(s+t_n)=\overline{m}\nonumber
\ee and the claim is proved.

\vskip0.5em 
\noindent 3. Let  $t_0>0$ be fixed;  as a consequence of step 2, there exists a point $x_0\in\omegb$ such that
\begin{equation}\label{maxi}
\max_{x\in\omegb}\ \tilde{z}(x,t_0)=\tilde{z}(x_0,t_0)=\overline{m}.
\end{equation}
To end this proof, it is enough to prove that (for some $t_0>0$) $x_0$ lies inside $\Omega.$ 

Indeed, if $x_0\in\Omega$, then the Parabolic Strong Maximum Principle and (\ref{maxi}) imply
\begin{equation}\label{maxi'}
\tilde{z}=\overline{m}\quad\hbox{ in }  \Omega\times(-\infty,t_0]
\ee
Therefore, by taking (\ref{maxi'}) into account, for any $x\in \overline{B}(\hat{x},r),$ it would follow that 
\begin{eqnarray}
\overline{m}=\tilde{z}(x,0)&=&\underset{\underset{n\to +\infty} {(y,s)\to (x,0)}}\limsup\ z_n(y,s)\nonumber\\
&=&\underset{\underset{n\to +\infty} {(y,s)\to (x,0)}}\limsup\ z(y,s+t_n)\nonumber\\
&=&\underset{\underset{n\to +\infty} {(y,s)\to (x,0)}}\limsup\ [z(y,s+t_n)-z(y,t_n)]+\underset{\underset{n\to +\infty} {y\to x}}\limsup\ z(y,t_n)\nonumber\\
&\leq&\underset{(y,s)\to (x,0)}\limsup\ \biggl[\|z(y,\cdot)\|_{C^{0,\frac{\nu}{2}}_{loc}(0,+\infty)}|s|^{\frac{\nu}{2}}\biggr] +\hat{z}(x)=\hat{z}(x)\nonumber
\end{eqnarray} with the last inequality following from Proposition~\ref{localholder}-(ii). Finally, since $\hat{z}$ is independent of the sequence $(t_n)_n$, we have obtained
 $$z(x,t)\to \overline{m}\quad  \hbox{uniformly on  }\overline{B}(\hat{x},r) \hbox{ as }t\to+\infty,$$
for any $\hat{x}\in \Omega$ and $r>0$ such that $\overline{B}(\hat{x},r)\subset\Omega,$  thus proving the statement of  Lemma~\ref{global-limits}.

\vskip0.5em 
\noindent 4.
Now, we are going to prove that, for some $t_0>0$,  $\tilde z(x,t_0)$ has a  maximum point $x_0$ in $\Omega.$ For that purpose, let us assume by contradiction that
\begin{equation}\label{contrary}
x_0\in \domeg \quad\hbox{ and }\quad \tilde{z}<\overline{m}\quad\hbox{ in } \Omega\times(0, t_0].
\end{equation} We pick some $0< \delta< t_0$ and argue in the subset $\Omega^\delta := \{x\in \Omega :0 < \dd(x) < \delta\} $ by introducing the function $\chi$ defined by 
\begin{equation*}
\chi(x,t):=\overline{m}+k(e^{-\dd(x)}-1)- k (t-t_0)\quad\hbox{ for all } (x,t)\in \overline{\Omega^\delta} \times[t_0-\delta, t_0] 
\end{equation*} where $\de$ and  $k$  are  to be chosen later in such a way to obtain
\begin{equation}\label{compare}
\tilde{z}\leq\chi\quad \hbox{ in } \quad\overline{\Omega^\delta} \times[t_0-\delta, t_0].
\end{equation}  
\vskip0.5em 
\noindent (i) We start by proving that $\chi$ satisfies $$\mathcal{L}\chi(x,t)>0\quad \hbox{ in } \Omega^\delta \times(t_0-\delta, t_0].$$
Computing we have  
\begin{eqnarray}
\mathcal{L}\chi(x,t)&=&- k \dd(x)+ke^{-\dd(x)}[\dd(x)\Delta \dd(x)-\dd(x)-\dd(x)\mathcal{H}(x)\cdot Dd(x)]\nonumber.
\end{eqnarray} By using (\ref{calH'}), we deduce that
$$\mathcal{L}\chi(x,t) = ke^{-\dd(x)}[-\dd(x) \, e^{\dd(x)}+ \dd(x)(\Delta \dd(x)  -1)+ m/(m-1)+o_\delta(1)]\,.$$ 
By the regularity of $\domeg$, $\Delta \dd$ is bounded in $\Omega^\delta$ and since $\dd(x)<\delta,$ we can choose $\delta>0$ small enough such that 
$$-\dd(x) \, e^{\dd(x)}+ \dd(x)(\Delta \dd(x)  -1)+ m/(m-1)+o_\delta(1)>0\; ,$$
for all $x$ such that $\dd(x)\leq \delta$. Hence we obtain
\begin{equation}\label{no-int-max}
\mathcal{L}\chi> 0\ \hbox{ on }\ \Omega^\delta\times (t_0-\delta,t_0]\, .
\end{equation} 
\vskip0.5em
\noindent (ii) For all $(x,t)\in\domeg\times (t_0-\delta,t_0],$ we have $\chi(x,t)=\overline{m}-k (t-t_0)\geq \overline{m}$ and knowing that $\tilde{z}\leq\overline{m}$ in $\omegb\times\mathbb{R},$ we easily conclude that 
\begin{equation}\label{ext-bd}
\tilde{z}\leq\chi\quad \hbox{ on } \domeg \times (t_0-\delta, t_0].
\end{equation} 
\vskip0.5em
\noindent (iii) We set  $\Gamma_\delta=\{x\in\omegb:\dd(x)=\delta\}.$ For any $(x,t)\in \Gamma_\delta\times(t_0-\delta ,t_0],$ we have $$\chi(x,t)=\overline{m}+k(e^{- \delta}-1)- k(t-t_0)\geq \overline{m}+k(e^{-\delta}-1).$$
Moreover, since $\tilde{z}(x,t)<\overline{m}$ in $\Omega\times (0, t_0],$ we have $\tilde{z}(x,t)<\overline{m} \hbox{ on }\Gamma_\delta\times[t_0-\de,t_0].$ We use the upper semi-continuity of $\tilde{z}$ to define 
$$
\eta(\delta):=\min_{\Gamma_\delta\times [t_0-\delta,t_0]}(\overline{m}-\tilde{z})
$$
and we find that $\tilde{z}(x,t)\leq\overline{m}-\eta(\delta) \hbox{ on }\Gamma_\delta\times [t_0-\delta,t_0].$ By choosing $k>0$ such that $k(e^{- \delta}-1)\geq-\eta(\delta),$ that is 
\begin{equation}\label{k}
0<k\leq\frac{\eta(\delta)}{1-e^{- \delta}},
\ee
one gets
\begin{equation}\label{int-bd}
\tilde{z}\leq\chi\quad \hbox{ on } \Gamma_\delta \times(t_0-\delta, t_0].
\end{equation} 
\vskip0.5em
\noindent (iv) For all $x\in \overline{\Omega^\delta},$ we have 
$$
\chi(x,t_0-\delta)=\overline{m}+k(e^{- \dd(x)}-1)+k \delta\geq \overline m + k (e^{-\delta}-1+\delta)\geq \overline m
$$ 
hence we deduce
\begin{equation}\label{initial-time}
\tilde{z}<\chi\quad \hbox{ on } \overline{\Omega^\delta} \times\{t_0-\delta\}.
\end{equation}
\vskip0.5em 
\noindent 5. 
Now we use  (\ref{halfrelaxed4}) in the interval $(t_0-\delta, t_0)$, together with  (\ref{no-int-max}), (\ref{ext-bd}), (\ref{int-bd}) and (\ref{initial-time}), so that   applying the comparison principle on $\overline{\Omega^\delta}\times[t_0-\delta,t_0]$  we finally conclude that (\ref{compare}) holds. 

Since $\tilde{z}(x_0,t_0)=\overline{m}=\chi(x_0,t_0),$ it follows from (\ref{compare}) that $\tilde{z}-\chi$ achieves its global maximum on $\overline{\Omega^\delta}\times[t_0-\delta,t_0]$
at $(x_0,t_0).$ By using the definition of the viscosity subsolution $\tilde{z}$ of (\ref{halfrelaxed4}), it necessarily follows that 
\be\label{bd}
\min\{\mathcal{L}\chi(x_0,t_0), \overline{m}+K\}\leq 0.
\ee But $\mathcal{L}\chi(x_0,t_0)=\frac{k m}{m-1}>0$ and $\overline{m}+K>0,$ hence we reach a contradiction with (\ref{bd}) and we conclude that (\ref{contrary}) does not hold. Hence, we have that \rife{maxi} holds for some $t_0>0$ and  $x_0\in \Omega$, and by Step 3 we deduce  that  (\ref{maxi'}) holds and the end of the proof follows.

\end{proof}  

\vskip1.0em

\section{The  Non-Convergence Cases: $c=0$  or $c>0$ and $1<m\leq 3/2$}\label{nonconvergence}

The main result of this section is 
\begin{theor}\label{subsolution} Let $\Omega$ be star-shaped with respect to a point $x_0\in\R^N$. Assume\footnote{this assumption is for instance verified when $f\in W^{1,\infty}(\Omega)$  and $$f(x)+c+ \|D f\|_\infty\, {\rm diam}(\Omega)\leq -\de$$} that  there exist  constants $\de>0$ and $\mu_0>0$ such that we have, for any $r\in [1-\mu_0,1)$: 
\begin{equation}\label{nega}
f(x)+c+ \frac{|f(x_0+r(x-x_0))-f(x)|}{(1-r)}\leq -\de
\quad\hbox{ on} \quad\omegb.
\end{equation} 
If $c=0$ and $1<m\leq 2$,  or  if $c> 0$ and $1<m\leq \frac32$,  then 
\begin{equation}\label{nonconvergence1}
u(x,t)+ct\to -\infty \quad\hbox{ as } t\to + \infty \,\quad \hbox{ locally uniformly in $\Omega$. } 
\end{equation} More precisely, there exist  a solution $\vfi$ of \rife{vhje-sub} and continuous functions $r(t)$, $H(t)$ such that
\begin{equation}\label{tesiexa}
u(x,t)+ct\leq r(t)^{\frac{2-m}{m-1}}\vfi(x_0+ r(t)(x-x_0)) - H(t)
\end{equation}
where $r(t)\uparrow 1$ and $H(t)\to + \infty$ as $t\to+\infty$ with the following rate
\begin{itemize}

\item[(a)] If $c>0$ then 
\begin{equation}\label{H-behavior}
\begin{cases}
H(t)=O(t^{\frac{3-2m}{2-m}})& \hbox{ if } 1<m<\frac32,\\
H(t)=O(\log t)& \hbox{ if } m=\frac32.
\end{cases}
\end{equation}

\item[(b)] If $c=0$ then 
 \begin{equation}\label{H-behavior2}
\begin{cases}
H(t)=O(t^{2-m})& \hbox{ if } 1<m<2,\\
H(t)=O(\log t)& \hbox{ if } m=2.
\end{cases}
\end{equation}
\end{itemize}
\end{theor}
\vskip0.5em

Let us recall that the ergodic constant $c$ depends itself on $f$, so that the assumptions made in the above theorem are not obviously checked. However,  such  assumptions  can actually happen to be true. To fix the ideas, consider the following Example which provides a specific case where condition (\ref{nega}) holds.

\begin{remark}\label{f-const}
Assumption \rife{nega} is always verified when $f$ is a constant.

We first consider the case that $f=0$ and denote by $c_0$ the corresponding ergodic constant. One can easily deduce that $c_0<0$. This is a  consequence of  the characterization (\ref{const-ergodic}) in Proposition~\ref{existencecondition}: indeed the constant functions are subsolutions of (\ref{additive-equation}) with $\lambda = 0$ while Proposition~\ref{existencecondition} yields that the infimum in (\ref{const-ergodic}) is not attained.
Consider now the case that $f=f_0$ is  a constant, possibly different from zero. Of course the   corresponding ergodic constant now is $c=c_0-f_0$. For every value of $f_0$,  \rife{nega} is clearly verified with any  $\de<|c_0|$. 


Next, it is not difficult to construct some function $f$ which is not constant and such that \rife{nega} is   verified. In particular, a small perturbation (in Lipschitz norm) of the constant $f_0$  still verifies \rife{nega}; indeed it is easy to check that the ergodic constant $c$ depends continuously on $f$ (with respect to perturbations in the sup-norm) as a consequence of formula (\ref{const-ergodic}).
\end{remark}

\vskip0.5em
\begin{proof}[\bf Proof of Theorem~\ref{subsolution}] Without loss of generality, assume that $x_0=0$. Moreover observe that we can always replace $\de$ with a smaller value in \rife{nega} without loss of generality.

In order  to prove that (\ref{tesiexa}) holds,   it is sufficient to prove that $r(t)^{\frac{2-m}{m-1}}\vfi(r(t)x) - H(t)$ is a supersolution of $E(\Omega,f+c,g+ ct ,u_0)$, since then  the estimate (\ref{tesiexa}) will follow by an application of the comparison result. 
Let then $\vfi$ be a solution of (\ref{vhje-sub}), we define the function $\vfi_r$ on $\overline{\Omega}\times [0,+\infty)$ as follows: 
$$
\vfi_r(x,t)=r(t)^{\frac{2-m}{m-1}}\vfi(r(t)x)\,
$$
where $r(t)$ will be chosen below in a way that $0<r(t)<1$ and $r(t)\uparrow 1$ as $t\to+\infty$. In particular, we fix
$$
r(0)=1- \mu_0
$$ 
so that $r(t)\in [1-\mu_0,1)$ and \rife{nega} may be applied.  Moreover, by the uniqueness result  for the ergodic problem, we can represent $\vfi$ as    $\vfi=\vfi_0+L$ with  $\vfi_0$ being the unique solution of (\ref{vhje-sub}) such that $\min_{\overline{\Omega}}\vfi_0 =0$ and $L$  a constant to be fixed later. In particular, note that $\vfi\geq L$ since $\vfi_0\geq0$ on $\overline{\Omega}$.

\vskip0.3em
\noindent 1. Our first step is in $\Omega.$ It is easy to see that
$$
-\Delta \vfi_r+|D \vfi_r|^m=r^{\frac{m}{m-1}}\left(-\Delta \vfi+|D \vfi|^m \right)
$$ hence, with the notation $A(z)=z_t-\Delta z+|D z|^m,$ we get
\begin{equation}\label{ex1}
A(\vfi_r)= r(t)^{\frac m{m-1}}(f(r(t)x)+c)+r(t)^{\frac{2-m}{m-1}}r'(t)\left(\frac{2-m}{m-1}\,\frac{\vfi(r(t)x)}{r(t)}+D \vfi(r(t)x)\cdot x\right).
\end{equation} 
Using the result (\ref{pv}), we know that  $\frac{\partial \vfi}{\partial \tau(x)}=o\left(\frac{\partial \vfi}{\partial \nu(x)} \right)$ as $x\to \domeg$, where $\tau$ and $\nu$ are tangential and normal vectors. Then, since $\Omega$ is  star-shaped with respect to $0$ and since  $D \vfi\cdot \nu>0$, we deduce the existence of $\sigma>0$ such that  $d(y)<\sigma$  implies $D \vfi(y)\cdot y\geq 0$.  Hence, since $\vfi$ is smooth inside $\Omega$, we have
$$
D \vfi(y)\cdot y\geq \chi_{\{d(y)>\sigma\}} \geq - C_\sigma
$$
for some constant $C_\sigma>0$. We deduce that
$$
D \vfi(r(t)x)\cdot x\geq  - \frac{C_\sigma}{r(t)}\,.
$$ 
In particular we have, for $L$ large enough,
$$  
\frac{2-m}{m-1}\,\frac{\vfi(r(t)x)}{r(t)}+D \vfi(r(t)x)\cdot x\geq  \frac1{r(t)}\left( \frac{2-m}{m-1}L-C_\sigma\right) >0.
$$
Going back to (\ref{ex1}), we have, using that $r(t)< 1$,
\begin{eqnarray*}
A(\vfi_r)&\geq&  r(t)^{\frac m{m-1}}(f(r(t)x)+c)\\
&=&r(t)^{\frac m{m-1}}(f(r(t)x)-f(x))+(r(t)^{\frac m{m-1}}-1)(f(x)+c)+f(x)+c\\
&\geq&- (1-r(t))\frac{| f(r(t)x)-f(x))|}{(1-r(t))}  + (r(t)^{\frac m{m-1}}-1)(f(x)+c)+f(x)+c\\
&\geq&-(1-r(t)^{\frac m{m-1}})\left[\frac{| f(r(t)x)-f(x))|}{(1-r(t))}+f(x)+c\right]+f(x)+c.
\end{eqnarray*}
Now we use (\ref{nega}) and we finally obtain
\begin{equation}\label{ex2}
A(\vfi_r)\geq  f(x)+c +  \de\left(1-r(t)^{\frac m{m-1}}\right) \,.
\end{equation} 
Setting 
$$
\tilde H(t)= \int_0^t\left(1-r(s)^{\frac m{m-1}}\right) ds\, ,
$$ 
it follows that $$A(\vfi_r(x,t)-\de \, \tilde H(t))\geq f(x)+c = A(u(x,t)+ct).$$ 

\vskip1em
\noindent 2.	Now, we consider the boundary $\domeg$ and  we turn to the choice of $r(t)$, which depends on the different values of $m$ and $c$. 

\begin{itemize}
\item[(i)]  When $c>0$ and $1<m<\frac32$. First observe that from (\ref{expansion4}), when $x\in \domeg$, we have that
\begin{equation}\label{boundarychoice}
\vfi(r(t)x)\geq K^*(1-r(t))^{-\frac{2-m}{m-1}}+L\,
\end{equation}
for some positive constant $K^*$.
 Then we choose $r$ such that 
\begin{equation*}
1-r(t)=\left(\frac{(c+\lambda)t+K^*\mu_0^{-\frac{2-m}{m-1}}}{K^*}\right)^{-\frac{m-1}{2-m}} \quad\hbox{ for all } t\geq0
\end{equation*} where   $\lambda>0$ will be determined later on. Note that $r(t)$ is increasing, $r(0)=1-\mu_0$, and $r(t)\uparrow 1$ as $t\to +\infty$. On the boundary, we get
$$
\vfi_r(x,t)\geq (1-\mu_0)^{\frac{2-m}{m-1}}\left( (c+\lambda)t+L+K^*\mu_0^{-\frac{2-m}{m-1}}\right),
$$
hence, up to choosing $\la$ sufficiently large (only depending on $c$ and $\mu_0$), we have
\begin{equation}\label{r-choice}
\vfi_r(x,t)\geq 2 ct+ (1-\mu_0)^{\frac{2-m}{m-1}}\left( L+K^*\mu_0^{-\frac{2-m}{m-1}}\right)\,.
\end{equation}
Next, since 
$$
\tilde H(t)= \int_0^t \left(1-r(s)^{\frac m{m-1}}\right) ds\leq K\, t
$$
for some $K$ depending on $c$, $K^*$, $\mu_0$, $\la$, and  since the boundary datum $g$ is bounded on $\domeg$, we have from \rife{r-choice}
$$
\vfi_r(x,t)\geq g(x)+ct+ \de\, \tilde H(t)
$$
up to taking in \rife{nega} some $\de$ eventually smaller and choosing  $L$ large enough.
We conclude that  $$\vfi_r(x,t)-H(t)\geq u(x,t)+ct\quad\hbox{ for all } x\in \domeg,\ t\geq0\,,
$$ 
where $H(t)= \de\, \tilde H(t)$.
\vskip0.4em
\item[(ii)] When $c=0$  and $1<m<2$,  we set
$$
1-r(t)= (\mu_0^{-\frac1{m-1}}+  t)^{-(m-1)}
$$
As before, $r(t)$ is increasing, $r(0)=1-\mu_0$ and $r(t)\uparrow 1$.  Using \rife{boundarychoice} we have, for every $x\in \domeg$,
$$
\vfi_r(x,t) \geq r(0)^{\frac{2-m}{m-1}}\vfi(r(t)x)\geq (1-\mu_0)^{\frac{2-m}{m-1}}\, \left(K^* (\mu_0^{-\frac1{m-1}}+  t)^{2-m}+L\right)\,,
$$
while 
$$
\tilde H(t)= \int_0^t \left(1-r(s)^{\frac m{m-1}}\right) ds\leq K\, \int_0^t \left(1-r(s)\right) ds = \frac K{2-m} \, (\mu_0^{-\frac1{m-1}}+  t)^{2-m}\,.
$$
Therefore, choosing $\de$ eventually smaller in \rife{nega},  and then choosing $L$ sufficiently large, we obtain again
\be\label{compg}
\vfi_r(x,t)\geq g(x)+ \de\tilde H(t)\quad\hbox{ for all } x\in \domeg,\ t\geq0\,,
\ee
which means that $\vfi_r(x,t)\geq u(x,t)+H(t)$ on  the lateral boundary with $H(t)= \de\,\tilde H(t)$.
\vskip0.4em
\item[(iii)] When $c=0$ and $m=2$. Here recall that we have
$$
\vfi(r(t)x)\geq -\log (1-r(t))- K^*+L\,
$$
for some positive constant $K^*$. We set now  
$$
1-r(t)= \frac1{t+\mu_0^{-1}}\,,
$$
hence
$$
\vfi_r(x,t) \geq \log (t+\mu_0^{-1})-K^*+L\,.
$$
Since 
$$
\tilde H(t)= \int_0^t \left(1-r(s)^{\frac m{m-1}}\right) ds\leq K\, \int_0^t \left(1-r(s)\right) ds= K\, \log(t+\mu_0^{-1})\,,
$$
we get again \rife{compg} by choosing $\de$ sufficiently small and $L$ large enough.
\end{itemize}

\vskip1em
\noindent 3.	Finally, at $t=0,$ we have
$$
\vfi_r(x,0)=r(0)^{\frac{2-m}{m-1}}\vfi(r(0)x) =(1-\mu_0)^{\frac{2-m}{m-1}}\vfi((1-\mu_0) x)\geq L(1-\mu_0)^{\frac{2-m}{m-1}} 
$$
hence for $L$  large we obtain $\vfi_r(x,0)\geq u_0(x)$ for all $x\in \omegb$.  

\vskip1em
\noindent 4. We conclude from the above comparison that $\vfi_r(x,t)\geq u(x,t)+ct+ H(t)$.  Now, according to the value of $c$ and $m$,  the choice of $r(t)$ in (i)--(iii) gives the rates (\ref{H-behavior}) or (\ref{H-behavior2}) claimed for $H$.
\end{proof}

\section{Appendix}

In this section, we will give the proofs of Theorem~\ref{supersolution} and Proposition \ref{localholder}-(i). 
\vskip0.2em
\begin{proof}[\bf Proof of Theorem~\ref{supersolution}]

\vskip0.2em
1. Let  $ d(x)$  be the signed distance function, which is negative when $x\not\in \overline\Omega.$ Let us fix  $ \de_0>0$ such that  $  d(x)$ is $C^2$ in $\{x\in \R^N: |  d(x)|<  2\de_0\}$ and set
$$
\tilde d(x)= \chi(  d(x))\,,
$$ 
where $\chi(s)$ is a smooth,  nondecreasing function such that
$\chi(s)=s$ for $0<|s|<\frac{\de_0}2$ and $\chi(s)$ is constant for $|s|>\de_0.$ Without loss of generality, we may have $\chi'(s)\leq 1$ for every $s.$ 
  Consider now the vector field
$$
n_k(x)=- \int_{\R^N} D \tilde d(y)\,\rho_k(x-y)\,dy
$$
where $\rho_k$ is a standard mollifying kernel (supported in the ball $B_{\frac1k}(0)$). Recalling the definition of $\tilde d$, the field $n_k(x)$ is supported in a  neighborhood of $\partial \Omega$ of radius $\de_0+\frac1k$, and  we have $n_k\in C^\infty$ for $k$ large. Moreover,  using the properties of $d(x)$ and in particular that   $ d\in C^2$, we have
\be\label{nk1}
|n_k|\leq 1\,,\quad |Dn_k|\leq \|D^2 d\|_\infty\,,\quad |D^2 n_k|\leq k \, \|D^2 d\|_\infty\,.
\ee
Clearly $n_k$ is  an approximation of the normal vector $\nu(x)=-Dd(x)$, in particular
\be\label{nk2}
|n_k(x)+D \tilde d(x)|\leq \frac{\|D^2 d\|_\infty}k\,.
\ee
Then we consider the function
$$
v(x,t)=\ga(t) \vfi(x-\mu(t)\, n_k(x))\,\qquad \hbox{ where $\vfi(x)= \vfi_0(x)- L.$ } 
$$
Here   $\vfi_0$ is the unique solution of \rife{vhje-sub} such that $\min\vfi_0=0$, and  $L$  is an additive constant to be chosen, whereas  $\ga(t) $ and $\mu(t)$ are positive functions, with values in $(0,1)$, that   will be fixed later in a  way that $\ga(t)\uparrow 1$ and $\mu(t)\downarrow 0$ as $t\to+\infty.$  
Observe that since 
$$
d(x-\mu(t)\, n_k(x))=d(x)- \mu \,D d(x)\cdot n_k(x) + O(\mu^2),
$$ by choosing  $k \geq 2\|D^2 d\|_{\infty}$ it follows that for any $x\in \Omega$ such that  $d(x)\leq \frac{\de_0} 2,$ we have, by using \rife{nk2} and the definition of $\tilde d,$ that:
 \be\label{std}
d(x)+ \frac12 \mu + O(\mu^2)\leq d(x-\mu(t)\, n_k(x))\leq d(x)+ \frac32 \mu + O(\mu^2).
\ee In the following, we fix $k$ as above. Moreover, in order to have   $\mu(t)$  sufficiently small,  it is enough to fix $\mu(0)$ small enough, since $\mu(t)$ is decreasing.  To fix the ideas, we set   $\beta=\mu(0),$ and we choose $\beta$ small enough so that, thanks to \rife{std}, we have the following:
\be\label{std2}
 \begin{array}{rl}
 &x-\mu(t)\, n_k(x)\in \Omega\quad\quad \forall \, x\in \Omega, t\geq 0, \\
 \noalign{\medskip}
 &d(x)+ \frac14 \mu(t)  \leq d(x-\mu(t)\, n_k(x))\leq d(x)+ 2 \mu(t)  \qquad \hbox{ in } \Omega^{\de_0/2}\times (0,+\infty).
 \end{array}
 \ee
 Note that this choice of $\beta$ only depends on $\de_0$ and $\|D^2d\|_\infty$, in other words only on the domain $\Omega$, and, eventually, we are allowed   to  take a smaller value of $\beta$ if needed later.
 \vskip1em

\vskip0.3em

2.	Let us compute now the equation for $v.$ Henceforth, we denote,  for any function $z$, 
$A(z):=z_t-\Delta z+|D z|^m$, and we use the letter $K$ to denote possibly different constants only depending on $\Omega$, $m$, $f$, $c.$

Since $D v= \ga(t)(I-\mu(t) Dn_k(x))D \vfi(x-\mu(t)\,n_k(x))$, using \rife{nk1}  and that $\ga,$ $\mu<1,$ we get 
$$
\begin{array}{c}
A(v)\leq -\ga \Delta \vfi + \ga^m\, (1+K\mu)^m |D \vfi |^m 
\\
\noalign{\medskip}+ \ga\, K \mu \left( |Dn_k(x)|\, |D^2\vfi| + |D\vfi|\right)  
+\ga'(t)\vfi -\mu'(t)\,\ga\, D\vfi \cdot n_k(x)
\end{array}
$$
where the argument of $\vfi$ is $x-\mu(t)n_k(x).$
Since $\vfi$ satisfies (\ref{vhje-sub}), we deduce that
\begin{equation}\label{uffa}
\begin{array}{c}
A(v)\leq \ga(f(x-\mu(t)n_k(x))+c)+\ga\left( \ga^{m-1}(1+K\mu)^m -1\right) |D \vfi|^m \\
\noalign{\medskip}
+ \ga\, K \mu \left( |Dn_k(x)|\,  |D^2\vfi| + |D\vfi|\right) +\ga'(t)\vfi -\mu'(t)\,\ga(t) D\vfi \cdot n_k(x).
\end{array}
\end{equation}
From the Lipschitz continuity of $f,$ we have
$$
\ga(f(x+\mu(t)n_k(x))+c)\leq (f(x)+c)  
 +(1-\ga(t)) \|(f+c)^-\|_{L^\infty(\Omega)}+\mu(t) \|D f\|_{L^\infty(\Omega)}.
$$
Moreover, since $\vfi=\vfi_0-L,$  using  (\ref{fi0}) for $\vfi_0$  and that $2-m<1$ we deduce that we have,   in the whole range $1<m\leq 2$,  
$$
\vfi \leq K (1+|D\vfi|)- L\qquad \forall x\in \Omega\,,
$$
for some constant $K>0$. We also use 
(\ref{Hessianfi2}), which we can suppose to hold true in the support of $n_k(x)$    without loss of generality. Therefore, we obtain from \rife{uffa}:
\begin{eqnarray*}
A(v)&\leq& f(x)+c +(1-\ga(t)) \|(f+c)^-\|_{L^\infty(\Omega)}+ \mu(t)\|D f\|_{L^\infty(\Omega)}\\
&&+ \ga\left( \ga^{m-1}(1+K\mu)^m -1+K\mu\right) |D \vfi|^m  \\
&&+\ga'(t)[K(1+|D \vfi|)-L] +\ga(t)\, K (|\mu'(t)|+ \mu(t))|D\vfi|
\,,
\end{eqnarray*} 
hence there exists a constant, still denoted by $K$, such that
\begin{eqnarray*}
A(v)&\leq& f(x)+c +(1-\ga(t)) \|(f+c)^-\|_{L^\infty(\Omega)}+ \mu(t)\|D f\|_{L^\infty(\Omega)}\\
&&+ \ga\left( \ga^{m-1}-1+K\mu\right) |D \vfi|^m +\ga'(t)[K(1+|D \vfi|)-L] \\
&&+\ga(t)\, K (|\mu'(t)|+ \mu(t))|D\vfi|.
\end{eqnarray*} 
Here we take $L>K$ and we choose  
$\ga(t)$ such that
\begin{equation}\label{gam}
\ga(t)^{m-1}=1-\la \mu(t)
\end{equation}
for some $\la$ large enough. Without loss of generality, we can assume that $\mu(t)$  is small in a  way that $\ga(t)>0$ (this amounts to ask $\beta=\mu(0)<\frac1\la$).  Then we obtain
\begin{eqnarray*}
A(v)&\leq& f(x)+c +(1-\ga(t)) \|(f+c)^-\|_{L^\infty(\Omega)}+ \mu(t)\|D f\|_{L^\infty(\Omega)} \,\\
&&+ \ga\mu\left(K-\la)\right |D \vfi|^m +\ga'(t)\,K \, |D \vfi|+\ga(t)\,K (|\mu'(t)|+\mu(t))| D\vfi | 
\end{eqnarray*} 
which yields,  by applying Young's inequality, 
\begin{eqnarray*}
A(v)&\leq& f(x)+c +(1-\ga(t)) \|(f+c)^-\|_{L^\infty(\Omega)}+ \mu(t)\|D f\|_{L^\infty(\Omega)} \,\\
&&+ \frac12\ga\mu(K-\la) |D \vfi|^m    +K\,\ga\,\mu\,  \left( \frac{\ga'(t)}{\ga\,\mu}\right)^{\frac{m}{m-1}} +K \,\ga\,\mu\,  \left(\frac{|\mu'|+\mu}{\mu}\right)^{\frac m{m-1}} 
\end{eqnarray*} 
hence, choosing $\la>K$ we get
\begin{eqnarray}\label{last}
A(v)&\leq& f(x)+c +(1-\ga(t)) \|(f+c)^-\|_{L^\infty(\Omega)}+ \mu(t)\|D f\|_{L^\infty(\Omega)} \nonumber\\
 &&+K\,\ga\,\mu\,  \left( \frac{\ga'(t)}{\ga\,\mu}\right)^{\frac{m}{m-1}}+K \,\ga\,\mu\,  \left(\frac{|\mu'|+\mu}{\mu}\right)^{\frac m{m-1}} 
\,.
\end{eqnarray}

3. Let us consider now the boundary. Since we have from \rife{std2}
$$
d(x-\mu(t)n_k(x))\geq \frac14 \mu(t) \quad\hbox{ for all } x\in \domeg\,,
$$
using the asymptotic behavior \rife{expansion4}--\rife{expansion4'}  for $\vfi_0$ we deduce that there exists a  constant $K^*$ such that
$$
\vfi_0(x-\mu(t)n_k(x))\leq K^*(\mu(t)^{-\al}+1)-L \quad\hbox{ for all } x\in \domeg
$$ 
when $\al=\frac{2-m}{m-1}>0$ (i.e. $m<2$), while when $m=2$ we have 
$$
\vfi_0(x-\mu(t)n_k(x))\leq - \log \mu(t)+ K^*-L \quad\hbox{ for all } x\in \domeg\,.
$$
We are going to take later $L$ sufficiently large. In particular, considering $L>K^*$ we deduce that for every $x\in \partial \Omega$ and $t>0$:
\begin{equation}\label{boun}
\begin{cases} 
\hbox{if $1<m<2$,} &\quad v(x,t)\leq K^* \mu(t)^{-\al}+ \ga(0)(K^*-L)\quad  \\
\hbox{if $m=2$,} &\quad v(x,t)\leq -\log \mu(t)+\ga(0)(K^*-L) 
\end{cases}
\end{equation} 
Moreover,  for the initial condition we have
\begin{equation}\label{ii}
v(x,0)=\ga(0)\vfi_0(x-\mu(0) n_k(x))-\ga(0) L \,.
\end{equation}
Let us now distinguish the choice of $\mu(t)$ according to different situations:
\begin{itemize}

\item[(i)] If $c>0$ and $1<m<2$, we set
$$ \mu(t)=\left( \frac{ct+K^*\beta^{-\al}}{K^*}\right)^{-\frac1\al}\,,
$$
where $\alpha=\frac{2-m}{m-1}.$ The value of  $\beta=\mu(0)$ has been already   chosen, as explained before, in order that \rife{std2} holds true and  also $1-\la\,\beta>0$, this choice only depends on   $\Omega$, $f$, $m.$ 
Note that $\mu(t)$ is a decreasing function and satisfies 
$$
K^*\mu(t)^{-\al}=ct+K^*\beta^{-\al}\,,$$ 
hence   (\ref{boun}) implies
\begin{equation*}
v(x,t)\leq ct +K^*\beta^{-\al}+\ga(0)(K^*-L)\quad\hbox{ for all }x\in \domeg\,.
\end{equation*}
Since $\ga(0)^{m-1}=1-\la\,\mu(0)=1-\la \beta>0$, and  since $g$ is bounded in $\domeg$, up to choosing $L$ sufficiently large we will have 
\begin{equation}\label{boun2}
v(x,t)\leq g(x)+ct\quad\hbox{ for all }x\in \domeg.
\end{equation}
Similarly, 
from the boundedness of $u_0$
on $\omegb,$ up to choosing $L$ large enough (again depending on $\beta$), we obtain from \rife{ii}
\begin{equation}\label{in}
v(x,0)\leq u_0(x) \quad\hbox{ for all } x\in \omegb.
\end{equation}
Consider now the behavior of $\mu(t)$, which implies that $\mu(t)\to 0$, and $\mu'(t)= o(\mu(t))$ as $t\to +\infty.$ 
By definition of $\ga$ in \rife{gam}, we have  that $\ga(t)\uparrow1,$ $1-\ga(t)=O(\mu(t))$  and $\ga'(t)=O(|\mu'(t)|)$ as $t\to+\infty.$  In particular we have
\be\label{mum}
\ga\,\mu\,  \left( \frac{\ga'}{\ga\,\mu}\right)^{\frac{m}{m-1}} +\ga\,\mu\,  \left(\frac{|\mu'|+\mu}{\mu}\right)^{\frac m{m-1}} \leq K\, \mu
\ee
hence we deduce from \rife{last}

\be\label{inside}
\begin{array}{c}
A(v)\leq f(x)+c +K\,\mu(t)\left(1+ \|(f+c)^-\|_{L^\infty(\Omega)}+ \|D f\|_{L^\infty(\Omega)} \right)\,.
\end{array}
\ee
Therefore, if  we 
set 
$$
\overline K=K\left(1+ \|(f+c)^-\|_{L^\infty(\Omega)}+ \|D f\|_{L^\infty(\Omega)} \right)
$$
and we define 
$$
w=u+ct+\overline K\int_0^t\mu(s)ds\,,
$$
we can conclude that $A(w) \geq A(v).$ Moreover, from \rife{boun2} and \rife{in} we have that $v\leq w$ on the parabolic boundary. By the standard comparison result, it follows that 
\begin{equation}\label{comp}
u(x,t)+ct\geq v(x,t)-\overline K \int_0^t\mu(s)ds
\end{equation}
for all $(x,t)\in \omegb\times[0,+\infty).$ Since $\mu(t)= O\left(t^{-\frac1\alpha}\right)$ as $t\to +\infty$, the conclusion of  (\ref{tesi2}) follows according to the values of $\alpha$, i.e. of $m.$ Note in particular that $\mu(t)\in L^1(0,\infty)$ if and only if $\al<1$ which corresponds to $m>\frac32.$

\item[(ii)] If $c>0$ and $m=2$, we set 
$$
\mu(t)=\beta\,e^{-ct}\,,
$$
where $\beta=\mu(0)$ is chosen as before. With this choice we have from \rife{boun}
$$
v(x,t)\leq ct -\log\beta+\ga(0)(K^*-L)\,,
$$
and  choosing $L$ large enogh we deduce \rife{boun2}. Of course  \rife{in} remains true as before. Finally,  in this case we have $\mu'(t)= O(\mu(t))$ as $t\to +\infty$, and then $\ga'(t)=O(\mu(t))$ as well; therefore \rife{mum} still holds true. Thus we obtain again   \rife{inside} and we conclude as before the inequality     \rife{comp}. Being $\mu$ integrable in $(0,+\infty)$, this implies \rife{tesi2} for the case $m=2.$

\item[(iii)] If $c=0$ and $1<m<2$, we set $\mu(t)$ as
$$
\mu(t)= (\beta^{-(1+\al)}+ \Lambda\, t)^{-\frac1{1+\al}}
$$
where $\Lambda$ will be  fixed later. The initial condition  $\mu(0)=\beta$  is fixed as before. We have
$$
\mu'(t)=  \frac{\Lambda\, \mu(t)}{\beta^{-(1+\al)}+\Lambda t}\leq \beta^{1+\al}\,\Lambda\,\mu(t)
$$
hence,  using also the definition of $\gamma$,  we obtain \rife{mum} with a constant depending on $\Lambda.$
We deduce then  from \rife{last} (recall that here $c=0$)
\be\label{cada}
\begin{array}{c}
A(v)\leq f(x) +  \overline K \,\mu(t) +K_\Lambda\,\mu(t)
\,,
\end{array}
\ee
where
$$
\overline K =K\, \left( 1+ \|(f)^-\|_{L^\infty(\Omega)}+ \|Df\|_{\infty}\right) 
$$
is a constant only depending on  $\Omega$, $m$, $f.$  In particular we obtain that
the function
\be\label{c=0}
w=u+(\overline K+K_\Lambda) \int_0^t\mu(s)ds
\ee
satisfies $A(w)\geq A(v).$ 
Now we choose $\Lambda=\frac{\overline K\,(1+\al)}{K^*\,\al}$, in order to have
\begin{eqnarray}
K^*\mu(t)^{-\al}&=& K^*\,(\beta^{-(1+\al)}+ \Lambda\, t)^{\frac\al{1+\al}}\nonumber\\
&=& K^*\Lambda\,\frac{\al}{1+\al}  \int_0^t \mu(s)\,ds+ K^* \beta^{-\alpha}\nonumber \\
&=& \overline K \int_0^t \mu(s)\,ds+K^*\beta^{-\alpha}\nonumber\\
&\leq& w(x,t)-u(x,t)+ K^*\beta^{-\al}.\nonumber
\end{eqnarray}
Therefore  from \rife{boun} we deduce that on the boundary
$$
v(x,t)\leq  w(x,t)- g(x) +K^*\beta^{-\al} +\ga(0)(K^*-L)\quad \hbox{ for all }  x\in \partial \Omega.
$$
 Choosing $L$ large enough we conclude that $v(x,t)\leq w(x,t)$ for every $x\in \partial \Omega$ and $t>0$; as before, we also have \rife{in}, hence by comparison we conclude that
$$
u(x,t)\geq v(x,t)-(\overline K+K_\Lambda) \int_0^t\mu(s)ds
$$
Since $\int_0^t \mu(s)\,ds= O\left( t^{\frac \al{1+\al}}\right)$ as $t\to +\infty$, we obtain \rife{tesi3} by definition of $\al.$

\item[(iv)] If $c=0$ and $m=2$ we set
$$
\mu(t)= \frac1{\Lambda \, t+\beta^{-1}}\,.
$$
Since we have
$$
|\mu'(t)|\leq \frac{\Lambda\,\mu(t)}{\Lambda \, t+\beta^{-1}}\leq \Lambda\, \beta\, \mu(t)
$$
we obtain \rife{cada} as before, hence   $A(w)\geq A(v)$ where $w$ is defined in \rife{c=0}. Since 
$$
-\log \mu=\Lambda\,  \int_0^t \mu(s)\,ds-  \log \beta
$$
we choose $\Lambda=\overline K$ and by using \rife{boun} we obtain on the lateral boundary: 
\begin{eqnarray}
v(x,t)&\leq&  \overline K\, \int_0^t \mu(s)\,ds -  \log \beta+ \ga(0)(K^*-L) \nonumber\\
&\leq& w(x,t)-g(x) -  \log \beta+ \ga(0)(K^*-L)\nonumber\,.
\end{eqnarray}
We conclude that $v\leq w$ on the boundary up to choosing $L$ large enough and then,  by comparison, we deduce that $v(x,t)\leq w(x,t)$ in $\Omega\times (0,T)$, which implies \rife{tesi3} for the case $m=2.$
\end{itemize}
\end{proof} 

\vskip0.5em
Now, we turn to the proof of the interior H\"older estimates on $v$ with respect to the $x$-variable uniformly in $t\in (0,+\infty)$ which is based on an idea introduced by Ishii and Lions \cite{IL}. This idea has been already used for instance in Barles \cite{BAR1} and Barles and Souganidis \cite{BS} to show gradient estimates of viscosity solutions to quasilinear elliptic ans parabolic PDE with Lipschitz initial conditions, by Barles and Da Lio \cite{BDL2} to prove local H\"older estimates up to the boundary of bounded solutions to fully non linear elliptic PDE with Neumann boundary conditions and by Da Lio \cite{DaLio1} to obtain $C^{0,\nu}$-estimates for viscosity solutions of parabolic equations with nonlinear Neumann-type  boundary conditions.
\vskip0.5em
\begin{proof}[\bf Proof of Proposition~\ref{localholder}-(i)]
We fix $ \nu\in (0,1),$ $\delta>0,$ $\eta>0$ and $x_0\in \Omega$ such that $\dd (x_0)\geq \delta.$ We are going to show that there exists a suitable constant $C$ depending on $\nu$, $\delta$, $\eta$,  $M:=\sup_{K_{\frac\delta 2}\times (\eta/2,+\infty)}|v|$ and the data of the problem such that, for all $y\in B(x_0,\delta /2)$ and $t_0\geq \eta,$ we have
\begin{equation}\label{holder}
v(x_0,t_0)-v(y,t_0)\leq C|x_0-y|^\nu \, . 
\end{equation} The property (\ref{holder}) clearly implies the $C^{0,\nu}$-estimates of $v(\cdot,t_0)$ in $K_\delta.$ Indeed, for all $x, y \in K_\delta,$ if $|x-y|\geq \delta/2$ then $$|v(x,t_0)-v(y,t_0)|\leq \frac{2^{\nu+1}M}{\delta^\nu}|x-y|^\nu.$$ But when $|x-y| < \delta/2,$ we can therefore apply (\ref{holder}) with $x_0=x$ and $y\in B(x,\delta/2)$ or with $x_0=y$ and $x\in B(y,\delta/2)$ and finally obtain the desired estimates.
\vskip0.5em

To prove (\ref{holder}), we consider the function $(x,y,t)\mapsto \Phi(x,y,t)$ defined on $\Omega \times \Omega  \times(0,+\infty)$ as follows
\begin{equation*}
\Phi(x,y,t)=v(x,t)-v(y,t)-\vfi(|x-y|)-L(|x-x_0|^2+|t-t_0|^2)
\end{equation*} where  $\vfi(t)=Ct^\nu$. The constants $L>0$ and $C>0$ will be chosen in such a way that $\Phi$ is a non-positive function.  

We first choose $L$ and $C$ in order to have 
$$\Phi\leq 0\quad\hbox{ on }\quad\partial \biggl(\overline{B}(x_0,\delta/2)\times \overline{B}(x_0,\delta/2)\times(\eta/2,+\infty)\biggr).$$ 
This leads to the constraints
$$
L\left(\frac{\delta}{4}\right)^2 \geq 2M\; ,\; L\left(\frac{\eta}{2}\right)^2 \geq 2M \; \hbox{and}\; C\left(\frac{\delta}{4}\right)^\nu \geq 2M\; .$$
With these choices of $C$ and $L,$ it is easily checked that $\Phi(x,y,t)\leq 0$ for any $x,y,t$ such that $|x-y|\geq \delta/4$ or $|x-x_0|\geq \delta/4$ or $|t-t_0|Ê\geq \eta/2$, and putting together these properties, we clearly have the desired property.
\vskip0.5em
Next, $L$ being fixed as above, we argue by contradiction, assuming that, for any constant $C$ (satisfying the above constraints), we have
\begin{equation}\label{maximum}
\mathcal{M}_{C,L}:=\underset{\overline{B}(x_0,\delta/2)\times\overline{B}(x_0,\delta/2)\times[\eta/2,+\infty)}\max\Phi(x,y,t) > 0.
\end{equation} Let $(\bar{x},\bar{y},\bar{t})\in\overline{B}(x_0,\delta/2)\times \overline{B}(x_0,\delta/2)\times [\eta/2,+\infty)$ be a maximum point of $\Phi;$ we have dropped the dependence of $\bar{x},\bar{y}$ and $\bar{t}$ in $C$ for sake of simplicity. By using (\ref{maximum}), it is clear that $\bar{x}\neq\bar{y},$ otherwise, we would have $\mathcal{M}_{C,L}\leq0.$ With the choices of $C$ and $L$ we made above, it is obvious that $(\bar{x},\bar{y},\bar{t})\in B(x_0,\delta/2)\times B(x_0,\delta/2)\times(\eta/2,+\infty).$ 

From (\ref{maximum}), we get 
\begin{equation}\label{C-large-loc}
C|\bar{x}-\bar{y}|^\nu+L(|\bar{x}-x_0|^2+|\bar{t}-t_0|^2)\leq v(\bar{x},\bar{t})-v(\bar{y},\bar{t})
\end{equation} which yields 
\begin{equation}\label{C-large}
C|\bar{x}-\bar{y}|^\nu\leq 2M \quad\hbox{ and }  \quad L(|\bar{x}-x_0|^2+|\bar{t}-t_0|^2)\leq 2M
\end{equation} and it follows, in particular, that $|\bar{x}-\bar{y}|\to0$ as $C\to+\infty$ and we recall that we may assume without loss of generality that $|\bar{x}-\bar{y}|>0$ for $C$ large enough.

We define $\phi$ by 
\begin{equation}\label{matrixlemma}
\phi(x-y)=\vfi(|x-y|)
\end{equation} and use the arguments of \cite[Proposition IV.1]{IL} to prove the existence of two $N\times N$ symmetric matrices  $B_1$ and $B_2$, and $a,b \in \R$ such that 
\begin{equation}\label{matriciallemma}
\left(
\begin{array}{cc}
  B_1-2LI & 0 \\
  0 & -B_2 \\
\end{array}%
\right)\leq
\left(
\begin{array}{cc}
  D^2\phi(\bar{x}-\bar{y})& -D^2\phi(\bar{x}-\bar{y}) \\
  -D^2\phi(\bar{x}-\bar{y}) & D^2\phi(\bar{x}-\bar{y}) \\
\end{array}%
\right) \; ,
\end{equation} $a- b - 2L(\bar{t}-t_0) \geq 0$ and 
\begin{equation}\label{viscosityineq}
\left\{
\begin{array}{rl}
a - Tr (B_1) + |p|^m -f(\bar{x})-c\leq0\\
b - Tr (B_2) + |q|^m -f(\bar{y})-c\geq0
\end{array}
\right. 
\end{equation}
where $$p= D\phi(\bar{x}-\bar{y})+2L(\bar{x}-x_0)\quad\hbox{ and }\quad q=D\phi(\bar{x}-\bar{y}).$$
For all $\xi,\zeta\in\R^N,$ we rewrite (\ref{matriciallemma}) as
\begin{equation}\label{matriciallemma1}
\langle B_1\xi,\xi\rangle-\langle B_2\zeta,\zeta\rangle\leq \langle D^2\phi(\bar{x}-\bar{y})(\xi-\zeta),\xi-\zeta\rangle +2L|\xi|^2.
\end{equation} Let $(e_i)_{1\leq i\leq N-1}$ be a familly of $(N-1)$ vectors in $\R^N$ such that $(e_1,e_2, ...,e_{N-1},\frac{ q}{|q|})$ is an orthonormal basis of $\R^N.$ Plugging successively $\xi=\zeta=e_i$ for all $i=1,...,N-1$ and $\xi=-\zeta=\frac{ q}{|q|}$ in (\ref{matriciallemma1}) and by adding all the inequalities obtained, we obtain:
\begin{eqnarray}\label{trace1}
Tr(B_1-B_2)&=&\sum_{i=1}^{N-1}\langle (B_1-B_2)e_i,e_i\rangle +\langle (B_1-B_2)\frac{ q}{|q|},\frac{ q}{|q|}\rangle\nonumber\\
&\leq&2NL+4 \frac{\langle D^2\phi(\bar{x}-\bar{y}) q,q\rangle}{|q|^2}.
\end{eqnarray} Going back to the form of $\phi,$ we set $\chi(z)=|z|$ and obtain 
\begin{equation}\label{trace2}
\langle D^2\phi(\bar{x}-\bar{y}) q,q\rangle=\vfi'\langle D^2\chi(\bar{x}-\bar{y}) q,q\rangle+\vfi''\langle (D\chi(\bar{x}-\bar{y})\otimes D\chi(\bar{x}-\bar{y}))  q,q\rangle.
\end{equation} But, knowing that $D\chi(z)=\frac{z}{|z|},$ it follows that $|D\chi(z)|^2=1$ which yields by differentiation $2D^2\chi(z)D\chi(z)=0$ and we obtain $D^2\chi(\bar{x}-\bar{y})q=0$ by taking $z=\bar{x}-\bar{y}.$ Therefore, we use (\ref{trace2}) and find that(\ref{trace1}) becomes
\begin{equation*}
Tr(B_1-B_2)\leq2NL+4\vfi'' \frac{\langle (q\otimes q)  q,q\rangle}{|q|^2}=2NL+4\vfi''\nonumber
\end{equation*} Finally, we obtain
\begin{equation}\label{trace3}
Tr(B_1-B_2)\leq 2LN+4C\nu(\nu-1)|\bar{x}-\bar{y}|^{\nu-2}.
\end{equation} By subtracting the two inequalities in (\ref{viscosityineq}) and using (\ref{trace3}), we have 
$$ 2L(\bar{t}-t_0)+4C\nu(1-\nu)|\bar{x}-\bar{y}|^{\nu-2} \leq  |q|^m-|p|^m+f(\bar{x})-f(\bar{y})+2LN\; .$$
Then the convexity of $p \mapsto |p|^m$ yields $|p|^m \geq |q|^m + m |q|^{m-2}q\cdot (p-q)$ and therefore
$$ |q|^m-|p|^m \leq m |q|^{m-1}|p-q| \leq m L\delta |q|^{m-1}\; ,$$
since $ |\bar{x}-x_0| \leq \delta/2$.

Plugging this estimate in the above inequality gives
$$
2L(\bar{t}-t_0)+4C\nu(1-\nu)|\bar{x}-\bar{y}|^{\nu-2}
\leq m L\delta |q|^{m-1} + 2 \|f\|_\infty  +2LN \; ,$$
and therefore
$$
2L(\bar{t}-t_0)+4C\nu(1-\nu)|\bar{x}-\bar{y}|^{\nu-2}
\leq m L\delta (\nu C|\bar{x}-\bar{y}|^{(\nu-1)})^{m-1} + 2 \|f\|_\infty  +2LN \; ,$$
Hence, using that $2L|\bar{t}-t_0| \leq 2L^{1/2}(2M)^{1/2}$, if $\tilde K := 2L^{1/2}(2M)^{1/2} + 2 \|f\|_\infty  +2LN$, we finally have
\begin{equation*}
4\nu(1-\nu)
\leq \tilde K \frac{|\bar{x}-\bar{y}|^{2-\nu}}{C}+mL\delta \nu^{m-1}C^{m-2}|\bar{x}-\bar{y}|^{(\nu-1)(m-2)+1}\; .
\end{equation*} 
But  $m-2\leq 0$ and $(\nu-1)(m-2)+1 \geq 1$ since $\nu-1\leq 0$ as well; from (\ref{C-large}), we know that $|\bar{x}-\bar{y}| \to 0$ when $C$ tends to infinity. Therefore it is clear  that this inequality cannot hold for $C$ large enough, thus contradicting (\ref{maximum}). Then,  if $C$ is large enough,  the estimate (\ref{holder}) holds true. Examining this proof, we find that $C$ depends on $\nu$, $\delta$, $\eta$, $M:=\sup_{K_{\frac\delta 2}\times (\eta/2,+\infty)}|v|$, $\|f\|_\infty$ and the constants $m$, $N$. 

\vskip0.5em

\end{proof}
\vskip0.5em
\noindent {\bf Acknowledgements.}  
This work was partially supported by the ANR ``Hamilton-Jacobi et thŽorie KAM faible'' (ANR-07-BLAN-3-187245), the AUF (Agence Universitaire de la Francophonie) scholarship program and the SARIMA (Soutien aux Activit\'es de
Recherche d'Informatique et de Math\'ematiques en Afrique) project.

The second author wishes to thank the ``Laboratoire de Math\'ematiques et Physique Th\'eorique'' of the ``Universit\'e de Tours'' for the invitation and the warm hospitality given on the occasion of this collaboration.

\end{document}